\documentclass[a4paper]{article}

\usepackage[english]{babel}
\usepackage[utf8x]{inputenc}
\usepackage[T1]{fontenc}

\usepackage[a4paper,top=3cm,bottom=2cm,left=3cm,right=3cm,marginparwidth=1.75cm]{geometry}

\usepackage{amsmath}
\usepackage{graphicx}
\usepackage{systeme}
\usepackage{mathrsfs} 
\usepackage{amsmath}
\usepackage{amssymb}
\usepackage{amsthm}
\usepackage{dsfont}

\usepackage{amssymb}
\usepackage{multicol}
\usepackage{nomencl}
\makenomenclature

\renewcommand\nompreamble{\begin{multicols}{2}}
\renewcommand\nompostamble{\end{multicols}}

\usepackage{float}
\usepackage{graphicx,subcaption}
\usepackage{color}
\usepackage{tabu}
\usepackage[ruled]{algorithm2e}
\usepackage{amsmath}
\usepackage[table]{xcolor}
\usepackage{array,multirow,graphicx}
 \usepackage{float}
 \usepackage{hhline}
\SetKwInOut{Input}{Input}
\SetKwInOut{Output}{Output}
\DontPrintSemicolon

\usepackage{subcaption}

\makeatother
\numberwithin{equation}{section}%
\theoremstyle{definition}
\newtheorem{definition}{Definition}[section]
\theoremstyle{theorem}
\newtheorem{theorem}{Theorem}[section]

\usepackage{url}
\usepackage[colorinlistoftodos]{todonotes}
\usepackage[colorlinks=true, allcolors=red]{hyperref}
\usepackage{float}
\everymath{\displaystyle}

\title{\textbf{Simultaneous identification of the parameters in the plasticity function for power hardening materials : A Bayesian approach}}
\author{ \, Salih Tatar \thanks{ Salih Tatar, Corresponding Author \hfil\break
Department of Mathematics $\&$  Computer Science, College of Science and General Studies\newline
Alfaisal University, Riyadh, KSA,\hfil\break
E-Mail : statar@alfaisal.edu} \, and 
Mohamed BenSalah\thanks{ Mohamed BenSalah,\hfil\break
Department of Computer Science, ISSAT of Sousse,\newline
University of Sousse, Rue Tahar Ben Achour, Sousse 4003, Tunisia\hfil\break
E-mail : mohamed.bensalah@fsm.rnu.tn}}

\begin{document}
\maketitle
\begin{abstract}
In this paper, we study simultaneous determination of the strain hardening exponent, the shear modulus and the yield stress in an inverse problem. First, we analyze the direct and the inverse problems. Then we formulate the inverse problem in the Bayesian framework. After solving the direct problem by an iterative approach,  we propose a numerical method based on a Bayesian approach for the numerical solution of the inverse problem. Numerical examples with noisy data illustrate applicability and accuracy of the proposed method to some extent.\\

\noindent \textbf{Keywords:} Bayesian inverse problems; iterative regularizing ensemble Kalman method; elastoplastic deflections; Ramberg-Osgood curve\\

\noindent \textbf{MSC 2010:} 35R30; 35J60; 65N21; 35Q74; 62F15
\end{abstract}

\nomenclature{$E:$}{Young’s modulus} 
\nomenclature{$\text{meas}(\Omega):$}{measure of the domain $\Omega$}
\nomenclature{$g:$}{plasticity function}
\nomenclature{$\nu:$}{Poisson coefficient}  \nomenclature{$\Omega:$}{cross section of a bar}
\nomenclature{$\partial \Omega:$}{boundary of $\Omega$}  \nomenclature{$\varphi:$}{angle of twist per unit length}
\nomenclature{$\xi_0^2:$}{yield stress} \nomenclature{$u(x,y):$}{Prandtl stress function}
\nomenclature{$G:$}{modulus of rigidity (shear modulus)} \nomenclature{$\xi^2:$}{stress intensity}
\nomenclature{$\mathcal{T}(\varphi):$}{theoretical value of the torque}
\nomenclature{$L^2(\Omega):$}{set of square integrable functions on $\Omega$}
\nomenclature{$\Delta_x:$}{mesh step in $x$ direction}
\nomenclature{$\Delta_y:$}{mesh step in $y$ direction}
\nomenclature{$\kappa:$}{strain hardening exponent}
\nomenclature{$\nabla:$}{gradient}
\nomenclature{$\mathbb{F}:$}{observation operator}
\nomenclature{$C:$}{variance matrix}
\nomenclature{$I:$}{identity matrix}
\nomenclature{$\varepsilon:$}{Gaussian noise}
\nomenclature{$N_e:$}{number of ensemble members}
\nomenclature{$d:$}{observed data}
\nomenclature{$\mathbb{P}:$}{probability operator}
\nomenclature{$\mathbb{P}(\cdot \vert \cdot):$}{Conditional probability}
\nomenclature{$\overline{\epsilon}:$}{tolerance parameter}
\nomenclature{$u_{ex}:$}{exact solution}
\nomenclature{$u^\star:$}{approximated solution}
\nomenclature{$e_{\kappa_n}:$}{relative error associated to $\kappa$}
\nomenclature{$e_{\xi_{0,n}}:$}{relative error associated to $\xi_0^2$}
\nomenclature{$e_{G_n}:$}{relative error associated to $G$}
\nomenclature{$e_{n}:$}{global relative error}
\nomenclature{$\theta^\dagger :$}{exact parameters}
\nomenclature{$\mathbb{P}(\theta):$}{prior distribution of the parameter $\theta$}
\nomenclature{$\gamma:$}{search step}
\nomenclature{$\rho:$}{control parameter}
\nomenclature{$\mathbb{G}:$}{set of admissible coefficients}
\nomenclature{$\overline{\theta}:$}{final estimate of the parameter $\theta$}
\nomenclature{$\|\cdot\|_{C[\xi_*^2, \xi^2]}:$}{Norm in the space $C[\xi_*^2, \xi^2]$}
\nomenclature{$C[\xi_*^2, \xi^2]$:}{space of continuous functions on $[\xi_*^2, \xi^2]$.}

\nomenclature{$\mu_0:$}{prior distribution}

\nomenclature{$\{\theta_0^j\}_{j=1}^{N_e}:$}{initial ensemble}
\nomenclature{$R_n:$}{residual at iteration $n$}
\nomenclature{$\delta:$}{noise level}
\nomenclature{$\mathcal{T}^\delta :$}{perturbed observational data}

\nomenclature{$\tau:$}{scaling factor used in the stopping criterion}

\nomenclature{$u^{(n)}:$}{solution at the 
$n-$th iteration}

\nomenclature{$\sigma:$}{standard deviation}
\nomenclature{$N(0,C):$}{Gaussian distribution with mean $0$ and covariance matrix $C$}

\nomenclature{$U(a_1,a_2):$}{uniform distribution over the interval $[a_1,a_2]$}

\nomenclature{$\Vert \cdot \Vert_{H^1}:$}{norm in the Sobolev space $H^1(\Omega)$}

\nomenclature{$\Vert \cdot \Vert:$}{euclidean norm}

\nomenclature{$\Vert \cdot \Vert_{2}:$}{$L^2$ norm}
\nomenclature{$\|\cdot\|:$}{euclidean norm}

\nomenclature{$M:$}{number of measurements}

\nomenclature{$N_x:$}{number of grid points along the 
$x$-axis}
\nomenclature{$N_y:$}{number of grid points along the 
$y$-axis}

\nomenclature{$u_{i,j}:$}{approximation of the solution $u$ at the grid point $(x_i,y_j)$}
\printnomenclature

\section{Introduction}
In this paper, we study an inverse coefficient problem for the following nonlinear boundary value problem that describes the quasi-static mathematical model of the elasto-plastic torsion of a strain hardening bar \cite{kachanov}: 

\begin{eqnarray}\label{govequation}
\left \{ \begin{array}{l}
-\nabla.\bigg(g(\vert \nabla u \vert^2) \nabla u \bigg)= 2\varphi, \quad (x,y)\in\Omega\subset \mathbb{R}^2, \\
u(x,y)=0, \quad  \quad  \quad  \quad  \quad  \quad  \quad  (x,y) \in \partial \Omega,
\end{array} \right.
\end{eqnarray}
where $\Omega:=(0,a)\times (0,b)$, $a, b > 0$ denotes the cross section of a cylindrical bar with piecewise smooth boundary $\partial \Omega$, $\varphi$ is the angle of twist per unit length, $|\nabla u|^{2}$ represents the stress intensity,  $g=g(\xi^2)$, $\xi^{2}=\vert \nabla u \vert^{2}$ is the plasticity function (since the relationship between the intensities of the shear stress and shear strain is given by the equation $S = g(\xi^2)\xi$, we use the notation $g(\xi^2)$ instead of $g(\xi)$ for plasticity function)  and $u(x,y)$ is the Prandtl stress function. According to deformation theory of plasticity, the plasticity function $g=g(\xi^2)$ satisfies the following conditions \cite{LA}: 
\begin{eqnarray}\label{S1-3}
\left\{ \begin{array}{lc}
c_1\leq g(\xi^2) \leq c_2,\\
c_1 \leq g(\xi^2) + 2g'(\xi^2) \xi^2 \leq c_3, \forall \xi^2 \in \big [{\xi_*}^2,{\xi^*}^2 \big],\\
g'(\xi^2)\geq 0,  \\
\exists \xi_0^2 \in ({\xi_*}^2,{\xi^*}^2), {\xi_*}^2  > 0:~g(\xi^2)=\frac{1}{G},~\forall \xi^2 \in \big [\xi_0^2, {\xi^*}^2 \big],
\end{array}\right.
\end{eqnarray}
where, $c_1,\,c_2,\, c_3$ are some positive constans, $ \xi_0^2 =\max_{(x,y) \in \Omega} \vert \nabla u(x, y)\vert^2$  is the yield stress which is the maximum stress or force per unit area within a material that can arise before the onset of permanent deformation. When stresses up to the yield stress are removed, the material resumes its original size and shape. In other words, there is a temporary shape change that is self-reversing after the force is removed, so that the object returns to its original shape. This kind of deformation is called pure elastic deformation. On the other hand, irreversible deformations are permanent even after stresses have been removed. One type of irreversible deformation is pure plastic deformation. For such materials, the yield stress marks the end of the elastic behavior and the beginning of the plastic behavior. In addition, $1/G$ is shear compliance, $G=E/(2(1+\nu))$ is the modulus of rigidity (shear modulus), $E>0$ is the Young's modulus,  $\nu \in (0,0.5)$ is the Poisson coefficient which is assumed to be 0.3 throughout this paper. For any angle $\varphi>0$, all points of the bar have non-zero stress intensity which means the condition ${\xi_*}^2 > 0$ in \eqref{S1-3} makes sense. It is also known that in order for the equation in (\ref{govequation}) to be elliptic, the first two conditions in (\ref{S1-3}) are needed. Furthermore, the last condition in (\ref{S1-3}) also makes sense because the elastic deformations precede the plastic ones. A set  $\mathbb{G}$  satisfying the conditions in (\ref{S1-3}) is called the class of admissible coefficients in optimal control and inverse problems theory. For many engineering materials, the function  $g(\xi^2)$ has the following form:
\begin{eqnarray}\label{curve}
g(\xi^2) = \left \{ \begin{array}{ll}
1/G, \quad \xi^2 \leq \xi_0^2,  \\
1/G~\left (\xi^2 / \xi_0^2 \right )^{0.5(1-\kappa)}, \quad  \xi_0^2 < \xi^2,
\end{array} \right .
\end{eqnarray}
which corresponds to the Ramberg-Osgood curve \cite{mamedov1, mamedov2, wei}. In (\ref{curve}), $ \kappa \in [0,1]$ is the strain hardening exponent. The values $\kappa=1$ and $\kappa=0$ correspond to pure elastic and pure plastic cases, respectively. When $\xi^2 > \xi_0^2$, the material is said to be undergoing plastic deformation, so large amounts of twisting can cause the beam to be permanently deformed. 
  Evidently, the function given by (\ref{curve}) is in the class of admissible coefficients $\mathbb{G}$. \\

For a given function $g(\xi^2)$ and the angle $\varphi$, the problem (\ref{govequation})  is called the direct (forward) problem. Like most direct problems of the mathematical physics, the direct problem (\ref{govequation}) is well-posed, i.e., the solution exists, unique and depends continuously on the input data, \cite{salih1, salih2, hasanov1}. \\

Next we define the inverse problem. The inverse problem consists of determining the pair of functions $\big \{ u(x,y), g(\xi^2) \big \}$ from the following nonlinear problem:
\begin{eqnarray}\label{inverseproblem}
\left \{ \begin{array}{l}
-\nabla.\bigg(g(\vert \nabla u \vert^2) \nabla u \bigg)= 2\varphi, \quad (x,y)\in\Omega\subset \mathbb{R}^2, \\
u(x,y)=0, \quad  \quad  \quad  \quad  \quad  \quad  \quad  (x,y) \in \partial \Omega, \\
2\int \limits_\Omega{u(x,y;\varphi_i)dx\,dy}=\mathcal{T}_i, ~~~i=1, \cdots, M,
\end{array} \right.
\end{eqnarray}
where 
\begin{eqnarray}\label{outputdata}
\mathcal{T}_i:= \mathcal{T} (\varphi_i ) = 2\int \limits_{\Omega} u(x, y ;\varphi_i)dx\,dy,
\end{eqnarray}
are the measured values of the torque (measured output data) corresponding to the angles $\varphi_i$, $i= 1, \cdots, M$ and $M > 1$ is the number of measurements. The inverse problem (\ref{inverseproblem}) has been studied in some papers both theoretically and numerically. For instance, in \cite{salih2}, a numerical reconstruction algorithm based on parametrization of the unknown coefficient $g(\xi^2)$  is proposed for the inverse problem.  The parametrization algorithm consists of discretization of the unknown curve (\ref{curve})  in the following form \cite{salih2}: 
\begin{eqnarray}\label{discretiazition}
g_h(\xi^2)=\left \{ \begin{array}{l} 
\beta_0=1/G,    \quad \quad \xi^2 \in \big (0,{\xi_0}^2 \big],\\
\beta_0 - \beta_1 \big(\xi^2-{\xi_0}^2 \big), \quad \quad \xi^2 \in \big ({\xi_0}^2, {\xi_1}^2 \big],\\
\beta_0-\sum_{i=1}^{M-1}\beta_i \big({\xi_i}^2 - {\xi_{i-1}}^2 \big) - \beta_M \big(\xi^2-{\xi_{M-1}}^2 \big),  \quad \xi^2 \in 
\big({\xi_{i-1}}^2, {\xi_{i}}^2 \big].
\end{array} \right.
\end{eqnarray}

The parametrization algorithm focuses on the unknown parameters $\beta_i$ appearing in (\ref{discretiazition}). At each $i$th state, the algorithm aims to find the parameters $\beta_i$, using the measured output data $\mathcal{T}_i$ that corresponds to the angle $\varphi_i$. Although the parametrization algorithm is used for numerical solution of some class of inverse coefficients problems, it has some disadvantages. The first one is that the application of this algorithm requires lots of measured output data. This is of course undesirable since getting these data needs time and costs. The second is the ill-posedness of the algorithm. This situation is illustrated in \cite{salih2} and a regularization method is offered. Based on those points a new method i.e., the semi-analytic inversion method is developed in \cite{salih3}. The semi-analytic inversion method is based on the determination of the three main unknowns $G,\, \xi_0^2, \, \kappa $, but not simultaneously. The first distinguishable feature of this algorithm is that it uses only a few values of the measured output data $(\varphi_i, \mathcal{T}_i)$. This method determines the unknown curve completely by using these a few data. The second distinguishable feature of this algorithm is that it is well-posed. In the semi-analytic inversion method, the algorithm for the determination of the yield stress is a bit complicated. A modification in this algorithm is suggested in \cite{salih4}  to deal with this issue. The existence of the solution to the inverse problem (\ref{inverseproblem}) is proved in \cite{salih5}. The paper \cite{salih6} is devoted to simultaneous determination of the parameters in (\ref{curve}), however, the problem considered in this paper is a parabolic problem and the numerical method given is based on discretization of
the minimization problem, steepest descent and least squares method. To the best of authors'  knowledge, there is no work deals with determination of the parameters in (\ref{curve})   simultaneously from the inverse problem (\ref{inverseproblem}).  In this paper,  we propose a numerical method based on a Bayesian approach for the determination of the unknown parameters in (\ref{curve}) simultaneously by using a few values of the measure output  data. This study can be regarded as continuation of the series of the works mentioned above.\\

This paper is organized as follows: In the next section, we prove that the corresponding Bayesian inverse problem is well-defined based on a proof of the continuity of the input-output map and that the posterior distribution depends continuously on the data. In Section \ref{bayesion}, we introduce the Bayesian method for estimating the unknown parameters in the considered inverse problem. Some numerical examples are given to show the efficiency of the method in section \ref{numerical}. 

\section{Analysis of the direct and the inverse problems}\label{analysis}

\begin{definition}
The weak solution $u \in \mathring{H}^{1}(\Omega)$ of the nonlinear boundary value problem (\ref{govequation}) is defined as the solution of the following problem: 
\begin{eqnarray}\label{govequationfunctional}
a(u;u,v)=l(v),\quad \forall v \in \mathring{H}^{1}(\Omega),
\end{eqnarray}
where the nonlinear and linear functionals $a(u;v,w)$ and $l(v)$ are defined by
\begin{eqnarray*}
a(u;v,w) =  \int \limits_{\Omega}   g\big(\vert\nabla u\vert^2 \big)\nabla v \nabla w \, dx \, dy, \quad 
l(v) = 2 \varphi \int \limits_{\Omega} v \, dx \, dy, \quad v \in \mathring{H}^{1}(\Omega),
\end{eqnarray*}
respectively, where $\mathring{H}^{1}(\Omega ):= \big \{ v \in H^1(\Omega ): ~ v(x)=0, ~x\in \partial \Omega \big \}$ and $H^1(\Omega )$ are the Sobolev spaces of functions. 
\end{definition}
It is proved in  \cite{hasanov1} that the direct problem (\ref{govequation}) has a unique solution $u \in \mathring{H}^{1}(\Omega)$. We can linearize the problem (\ref{govequationfunctional})  as follows by using the monotone iteration scheme introduced in \cite{hasanov1}  : 
\begin{eqnarray}\label{linearization}
a\big(u^{(n-1)};u^{(n)},v\big)=l(v),\quad \forall v \in \mathring{H}^{1}(\Omega).
\end{eqnarray}
The following theorem states that the solution of the linearized problem (\ref{linearization}) converges to the unique solution of the problem (\ref{govequationfunctional}), \cite{hasanov1, hasanov2}. 

\begin{theorem}
Let the conditions in (\ref{S1-3}) hold. Then the approximate solution $u^{ (n)} \in \mathring{H}^{1}(\Omega)$, defined by the iteration scheme \eqref{linearization}, of the nonlinear problem (\ref{govequationfunctional}) converges to unique exact solution $u\in \mathring{H}^{1}(\Omega)$ of the problem (\ref{govequationfunctional})   in $H^1$ norm, as $n \to \infty$. 
\end{theorem}

Next we reformulate the inverse problem (\ref{inverseproblem}). We denote by $u(x, y ; g ; \varphi)$ the solution to the direct problem (\ref{govequation}) for a given $g \in \mathbb{G}$ and $\varphi \in [\varphi_*, \varphi^*], \, \varphi_* > 0$. Then the inverse problem (\ref{inverseproblem}) can be defined as a solution of the following nonlinear functional equation:
\begin{eqnarray}\label{functionalequation}
2\int \limits_{\Omega} u(x, y ; g ; \varphi)dx\,dy =  \mathcal{T} (\varphi), \, g \in \mathbb{G}, \, \varphi \in [\varphi_*, \varphi^*].
\end{eqnarray}
By the definition (\ref{outputdata}), we define the input-output map $T : \mathbb{G} \to \bold {T} $ from the class of admissible coefficients to the class of output functions $T (\varphi) \in \bold {T} $. The we can reformulate the inverse problem (\ref{inverseproblem}) in terms of input-output map as the following nonlinear operator equation : 
\begin{eqnarray}\label{inputoutputequation}
 T(g) = \mathcal{T}, g \in \mathbb{G}.
\end{eqnarray}
We conclude from (\ref{inputoutputequation}) that the inverse problem considered here can be reduced into inverting the input-output map $T : \mathbb{G} \to \bold {T}$. The invertibility and thus existence of the solution to the inverse problem is studied in  \cite{salih5}. \\

The following theorem states that the input-output map defined by (\ref{inputoutputequation}) is continuous. 
\begin{theorem}\label{imp}
Let $u_1(x, y):= u(x, y ; g_1 ; \varphi)$ and $u_2(x, y):= u(x, y ; g_2 ; \varphi)$ be the solutions to the direct problem (\ref{govequation}) for given $g_ 1 \in \mathbb{G}$  and $g_ 2 \in \mathbb{G}$, respectively. Then the input-output map defined by (\ref{inputoutputequation}) is continuous and the following estimates holds:
\begin{eqnarray}\label{mapcontiniuty}
\Vert T(g_1) - T(g_2) \Vert_ {C[{\xi_*}^2 , {\xi^*}^2]}  \leq 2 \, C \,\Vert \nabla u_2\Vert_2 \, \Vert g_1 - g_2 \Vert _{C[{\xi_*}^2 , {\xi^*}^2]},
\end{eqnarray}
where $C > 0$ is a constant, $\Vert . \Vert_ {C({\xi_*}^2 , {\xi^*}^2)}$ is the maximum norm in the space of continuous functions and $\Vert . \Vert_ 2$ is the usual  $L_2$ norm. 
\end{theorem}
\begin{proof}
As $u_1$ and   $u_2$  $\in \mathring{H}^{1}(\Omega)$ are the solutions to (\ref{govequation}), by (\ref{govequationfunctional}), we have: 

\begin{eqnarray}\label{v1}
 \int \limits_{\Omega} g_1\big(\vert\nabla u_1\vert^2 \big)\nabla u_1 \nabla v \,dx\,dy =  2 \varphi  \int \limits_{\Omega}  v \,dx\,dy,
\end{eqnarray}
\begin{eqnarray}\label{v2}
 \int \limits_{\Omega}  g_2\big(\vert\nabla u_2\vert^2 \big)\nabla u_2 \nabla v \,dx\,dy =  2 \varphi  \int \limits_{\Omega} v \,dx\,dy, 
\end{eqnarray}
 $\forall v \in \mathring{H}^{1}(\Omega)$. Then replacing $v$ by $u_2 - u_1$ in  (\ref{v1}) and by $u_1 - u_2$ in (\ref{v2})  and then subtracting the resulting equations, we get: 
 \begin{eqnarray}\label{v3}
 \begin{split}
 \int \limits_{\Omega}  \bigg(g_1\big(\vert\nabla u_1\vert^2 \big) \nabla u_1 &  -  g_1\big(\vert\nabla u_2\vert^2 \nabla u_2 \bigg)   \nabla (u_1 - u_2) \,dx\,dy  =  \\
 &  \int \limits_{\Omega}  \bigg(g_1\big(\vert\nabla u_1\vert^2 \big) \nabla u_1 -  g_1\big(\vert\nabla u_2\vert^2 \nabla u_2 \bigg)   \nabla (u_1 - u_2) \,dx\,dy.
 \end{split}
\end{eqnarray}
By using the following inequality for the functions in $\mathbb{G}$ \cite{Arzualemdar}
 \begin{eqnarray*}
  \left[  g(\xi^2)\xi - g(\tilde{\xi}^2)\tilde{\xi}\right]\cdot(\xi - \tilde{\xi} )\geq c |\xi - \tilde{\xi}|^2 ,
 \end{eqnarray*}
on the left hand side in (\ref{v3}) and Cauchy-Schwarz (CS) inequality on the right hand side in (\ref{v3}), we deduce that
\begin{eqnarray}\label{v4}
c \big \Vert \nabla u_1 - \nabla u_2 \big \Vert_2 \leq \big \Vert   g_2\big(\vert\nabla u_2\vert^2 \big)\nabla u_2 - g_1\big(\vert\nabla u_2\vert^2 \big)\nabla u_2  \big \Vert_2.
\end{eqnarray}
Now we estimate $\vert T(g_1) - T(g_2) \vert$ by using Cauchy-Schwarz (CS) and Poincare (P) inequality as follows: 
\begin{equation}\label{v5}
\begin{split}
\vert T(g_1) - T(g_2) \vert  & \overset{\text{(\ref{outputdata})}} {=}  2 \bigg \vert \int \limits_{\Omega} u(x, y ; g_1 ; \varphi)dx\,dy - \int \limits_{\Omega} u(x, y ; g_2 ; \varphi)dx\,dy \bigg \vert \\
& \leq 2  \int \limits_{\Omega}  \bigg \vert u(x, y ; g_1 ; \varphi) -  u(x, y ; g_2 ; \varphi) \bigg \vert  dx\,dy  \\
& \overset{\text{(CS)}} {\leq} 2 \,  \text{meas} (\Omega) \, \Vert u_1 - u_2 \Vert_2  \overset{\text{(P)}} {\leq} 2\,  \text{meas} (\Omega) \, C_{\Omega} \Vert \nabla (u_1 - u_2)\Vert_2,
\end{split}
\end{equation}
where $C_{\Omega} > 0$ is the Poincare constant. Taking into account (\ref{v4}) in (\ref{v5}), we have (\ref{mapcontiniuty}) where $C = \frac {2\,  \text{meas} (\Omega) C_{\Omega}  }{c}$. 
\end{proof}

The following theorem is straightforward with Theorem \ref{imp} and Theorems 1.1 and 1.2 in \cite{Trillos}. 

\begin{theorem}
If the prior $\mu_0$ is any measure with $\mu_0 {(X)} = 1$, then the Bayesian inverse problem of recovering $\theta=(\kappa, \xi_0^2, G)$ from the data $ d = \mathbb{F}(\theta) + \varepsilon$ is well-defined and well-posed.
\end{theorem}

\section{Iterative regularizing ensemble Kalman method}\label{bayesion}
In this section, we propose a numerical technique for recovering the plasticity function in the inverse problem (\ref{inverseproblem}) by using the measured output data $\mathcal{T}(\varphi)$ defined by (\ref{outputdata}). In practice, the measured data  can only be given with some measurement error leading to inaccuracies in the numerical solution of the inverse problem. To address this challenge, some regularization methods are considered. A modified optimal perturbation algorithm  \cite{li2013simultaneous} was introduced to regularize the ill-posed problems. While this method provides point estimates for the parameters, our interest extends beyond more point estimations, we seek to understand their statistical properties. Bayesian inference offers an exact solution to this requirement by incorporating uncertainties from observational noise and prior information, yielding the posterior probability density of the unknown parameters. For a comprehensive understanding of the well-posedness and underlying theory of this approach, one can refer to \cite{chada2018parameterizations,cotter2010approximation,dashti2011besov,iglesias2016regularizing,iglesias2015filter,stuart2010inverse}. Furthermore, Yan et al. in \cite{yan2017convergence,yan2019adaptive,yan2019adaptive1} and Stuart et al. in \cite{dashti2011uncertainty} have explored various numerical methods for solving Bayesian inverse problems. We begin by establishing some notation. Let $\theta=(\kappa, \xi_0^2, G)$ represent the discrete input, and define the observation operator $\mathbb{F}$ as follows:
$$\mathcal{T}=\mathbb{F}(\theta).$$ Let $\varepsilon$ denote additive Gaussian noise for observations, i.e., $\varepsilon \sim N(0, C)$, where $C$ is a variance matrix represented by $C = \sigma^2 I$, where $I$ is the identity matrix and $\sigma$ represents the standard deviation. Consequently, the observation data can be expressed as:
\begin{equation} \label{mes-per}
 d = \mathbb{F}(\theta) + \varepsilon.  
\end{equation} 
Here, we assume that $\varepsilon$ is independent of $\theta$. In the Bayesian context, both $\theta$ and $d$ are treated as random variables. Bayes' rule yields the posterior probability density, given by:
\begin{equation}\label{prob}
 \mathbb{P}(\theta|d) \propto \mathbb{P}(d|\theta)\mathbb{P}(\theta)   
\end{equation} 

Here, $\mathbb{P}(\theta) = \mathbb{P}(\kappa) \otimes \mathbb{P}(\xi_0^2) \otimes \mathbb{P}(G)$ represents the prior distribution before observing the data, where $\mathbb{P}(\kappa) = U(a, b)$, $\mathbb{P}(\xi_0^2) = U(c, d)$ and $\mathbb{P}(G) = U(e, f)$. Additionally, $\mathbb{P}(d|\theta)$ denotes the likelihood defined by:
\begin{equation*}
  \mathbb{P}(d|\theta) \propto \exp\left( -\frac{1}{2} \|C^{-\frac{1}{2}}(d - \mathbb{F}(\theta))\|^2 \right).  
\end{equation*} 
Obtaining the posterior distribution marks the initial stage of the Bayesian inverse problem. The subsequent challenge lies in extracting the desired information from the posterior probability density. Typically, statistical information from the posterior density is derived through sampling methods. A commonly used approach for sampling the posterior distribution is based on Markov Chain Monte Carlo (MCMC). However, this approach often demands a large number of evaluations for solving the forward model, which can be impractical for large-scale problems. Therefore, it may be infeasible to fully exploit the information provided by the posterior probability density function. Nevertheless, sampling can serve as a valuable tool for benchmarking and assessing methods that are feasible for practical application. This holds true for ensemble methods, which have emerged as an excellent approach for capturing features of the posterior distribution $\mathbb{P}(\theta|d)$ within a reasonable computing time.\\

In this paper, we extend the iterative regularizing ensemble Kalman method ({\bf IREKM}) introduced in \cite{iglesias2015iterative,zhang2018bayesian} to solve the inverse problem (\ref{inverseproblem}). Algorithm \ref{algo1} summarizes the pseudocode of the {\bf IREKM}. In Algorithm \ref{algo1}, we observe that the ensemble Kalman method offers a derivative-free optimization approach for solving the inverse problem, overcoming the need for severe adjoint problem calculations in gradient-based methods. This characteristic makes it particularly advantageous to deal with the problems where computing gradients is challenging. Consequently, the method has gained significant attention. For a deeper exploration of numerical theory, references \cite{iglesias2015iterative,evensen2018analysis,iglesias2013ensemble} provide valuable insights. While Schillings et al. in \cite{schillings2018convergence} offers a convergence analysis for linear inverse problems, the convergence behavior of {\bf IREKM} remains an open question for nonlinear inverse problems \cite{iglesias2015iterative}. Hence, in this paper, our focus remains on the numerical implementation of the inverse problem considered in this paper. 
\section{Numerical experiments}\label{numerical}
This section is concerned with demonstrating the efficiency and accuracy of Algorithm \ref{algo1}, as outlined in the previous section, in simultaneously identifying the strain hardening exponent \(\kappa\), the yield stress $\xi_0^2$ and the shear modulus \(G\) in the inverse problem (\ref{inverseproblem}) from the observed data \(\mathcal{T}(\varphi)\). First we solve the direct problem, where we generate synthetic data based on known values of $\kappa$, $\xi_0^2$  and \(G\). This allows us to simulate real-world scenarios and evaluate the performance of our algorithm under controlled conditions. By comparing the reconstructed parameters with the ground truth values used to generate the synthetic data, we can assess the accuracy of our identification method. Additionally, we investigate the computational efficiency of Algorithm 1 by analyzing its runtime and memory usage, providing insights into its scalability and applicability to larger-scale problems. Through this comprehensive evaluation, we aim to establish the robustness and reliability of our approach for solving the inverse problem posed by the nonlinear model (\ref{govequation}).

\vspace{0.2cm}
\begin{algorithm}[]
\vspace{0.2cm}
\begin{enumerate}
\item[{\bf 1.}] {\bf Initialisation:}
Generate an initial ensemble of $N_e$ members $\{\theta_0^j\}_{j=1}^{N_e}$ from the prior distribution $\mathbb{P}(\theta)$. Then, for each $j \in \{1, 2,\dots , N_e\}$, let $d_j=d+\zeta^j$ where $\zeta^j \sim N(0, C)$. Set $n=0.$
\item[{\bf 2.}] {\bf Prediction:} Let $\theta_{n}^j = (\kappa_n^j, \xi_{0,n}^{2,j}, G_{n}^j)$, $j=1,\dots,N_e$, $n=0,1,\dots.$ Calculate
$w_{n}^j = \mathbb{F}(\theta_{n}^j),\quad \text{for}\quad j \in \{1, 2, \ldots, N_e\},$ and compute the ensemble mean $$\overline{w}_n = \frac{1}{N_e} \sum_{j=1}^{N_e}w_{n}^j$$
\item[{\bf 3.}] \textbf{Analysis step:} Let 
    \begin{align*}
        C_{n}^{ww} &= \frac{1}{N_e - 1} \sum_{j=1}^{N_e} (\mathbb{F}(\theta_{n}^j) - \overline{w}_n)(\mathbb{F}(\theta_{n}^j) - \overline{w}_n)^T, 
         \\
         C_{n}^{\kappa w} &= \frac{1}{N_e - 1} \sum_{j=1}^{N_e} (\kappa_{n}^j - \overline{\kappa}_{n})(\mathbb{F}(\theta_{n}^j) - \overline{w}_n)^T
         \\
        C_{n}^{\xi_0^2 w} &= \frac{1}{N_e - 1} \sum_{j=1}^{N_e} (\xi_{0,n}^{2,j} - \overline{\xi}_{0,n}^2)(\mathbb{F}(\theta_{n}^j) - \overline{w}_n)^T,\\
        C_{n}^{G w} &= \frac{1}{N_e - 1} \sum_{j=1}^{N_e} (G_{n}^j - \overline{G}_{n})(\mathbb{F}(\theta_{n}^j) - \overline{w}_n)^T.
    \end{align*}
   where $$ \overline{\kappa}_n = \frac{1}{N_e} \sum_{j=1}^{N_e} \kappa_{n}^j, \quad \overline{\xi}_{0,n}^2 = \frac{1}{N_e} \sum_{j=1}^{N_e} \xi_{0,n}^{2,j}, \quad \overline{G}_n = \frac{1}{N_e} \sum_{j=1}^{N_e} G_{n}^j.$$
    Update each ensemble member:
    \begin{align*}
        \kappa_{n+1}^j &= \kappa_{n}^j + C_{n}^{\kappa w} (C_{n}^{ww} + \gamma_n C)^{-1} (d_j - w_{n}^j), \quad j \in \{1, 2, \ldots, N_e\},\\
       \xi_{0,n+1}^{2,j} &= \xi_{0,n}^{2,j} + C_{n}^{\xi_0^2 w} (C_{n}^{ww} + \gamma_n C)^{-1} (d_j - w_{n}^j), \quad j \in \{1, 2, \ldots, N_e\}, \\
        G_{n+1}^j &= G_{n}^j + C_{n}^{G w} (C_{n}^{ww} + \gamma_n C)^{-1} (d_j - w_{n}^j), \quad j \in \{1, 2, \ldots, N_e\}, 
    \end{align*}
    where $\gamma_n$ is chosen as follows:
Let $\gamma_0$ be an initial guess, and $\gamma_n^{i+1}=2^i \gamma_0$. Choose $\gamma_n=\gamma_n^N$
where $N$ is the first integer such that
    \[ \gamma_n^N \|C^{-\frac{1}{2}} (C_{n}^{ww} + \gamma_n^N C)^{-1} (d - \overline{w}_n)\| \geq \rho \|C^{-1} (d - \overline{w}_n)\|, \]
    and $\rho \in (0, 1)$ is a constant.
\item[{\bf 4.}] {\bf Iteration:} Increase $n$ by one and go Step 2, repeat the above procedure until a stopping criterion is satisfied.
\item[{\bf 5.}] {\bf Output:} Take $\overline{\theta}_n = (\overline{\kappa}_n, \overline{\xi}_{0,n}^2, \overline{G}_n)$ as the numerical solution of $\theta= (\kappa,  \xi_0^2 , G)$.

\end{enumerate}
    \caption{\it Iterative Regularizing Ensemble Kalman Method $(${\bf IREKM}$)$}
    \label{algo1}
\end{algorithm}

\subsection{Numerical solution of the direct problem}
This section is devoted to introduce an iterative approach for solving the direct problem (\ref{govequation}). Due to the natural complexity of solving nonlinear models numerically, we will employ a linearization technique as an efficient approach to deal with the challenges posed by nonlinearities. In doing so, we initialize $u^{(0)} \equiv 0$ and for $n = 1,2,\dots$, $u^{(n)}$ represents the solution to the following linear model:
\begin{equation*}
(\mathcal{LP}^{(n)})\left\{\begin{aligned}
-\nabla.(g(|\nabla u^{(n-1)}|)\, \nabla u^{(n)}) & =2\varphi &  \text{in}\quad & \Omega, \\
u^{(n)} & =0 &\text{on}\quad &   \partial \Omega, 
\end{aligned}\right. 
\end{equation*}
As it is mentioned in the section (\ref{analysis}), the solution of the forward problem (\ref{govequation}) can be obtained as the limit of the sequence $\{u^{(n)}\}_{n=0}^\infty$. To address the partial differential equation \(-\text{div}(g(|\nabla u^{(n-1)}|^2) \nabla u^{(n)}) = 2\varphi\) over the simplified computational domain \(\Omega =[0, a] \times [0,b]\), we employ the finite difference method with central difference approximations for the Laplacian and divergence terms. We set up a grid with nodes $(x_i, y_j)$ where $x_i = i \Delta_x$ and $y_j = j \Delta_y$ for $i = 0, \dots, N_x$ and $j = 0, \dots, N_y$. The mesh sizes are determined by $\Delta_x = a/N_x$ and $\Delta_y = b/N_y$. At each grid point \((x_i, y_j)\), the divergence of the gradient of \(u^{(n)}\), influenced by a spatial coefficient \(g(|\nabla u^{(n-1)}|^2)\), is approximated by:
\begin{equation*}
\begin{split}
-\text{div}(g(|\nabla u^{(n-1)}|^2) \nabla u^{(n)}) & \approx -\frac{1}{\Delta_x^2} \left[ \tilde{g}_{i+\frac{1}{2},j} (u_{i+1,j} - u_{i,j}) - \tilde{g}_{i-\frac{1}{2},j} (u_{i,j} - u_{i-1,j}) \right] \\
 & - \frac{1}{\Delta_y^2} \left[ \tilde{g}_{i,j+\frac{1}{2}} (u_{i,j+1} - u_{i,j}) - \tilde{g}_{i,j-\frac{1}{2}} (u_{i,j} - u_{i,j-1})\right],
\end{split}
\end{equation*}
for $i=1,\dots, N_x-1$ and $i=1,\dots, N_y-1$. Here, \( \tilde{g}_{i+\frac{1}{2},j} \), \( \tilde{g}_{i-\frac{1}{2},j} \), \( \tilde{g}_{i,j+\frac{1}{2}} \), and \( \tilde{g}_{i,j-\frac{1}{2}} \) represent the averaged values of \(\tilde{g}(x,y)=g(|\nabla u^{(n-1)}(x,y)|^2)\) at respective cell interfaces to maintain accuracy at the boundaries of each finite element, i.e,
$$ \tilde{g}_{i\pm\frac{1}{2},j}=\frac{\tilde{g}(x_{i\pm 1},y_{j})+\tilde{g}(x_{i},y_j)}{2}\; \text{and}\; \tilde{g}_{i,j\pm 1}=\frac{\tilde{g}(x_{i},y_{j\pm 1})+\tilde{g}(x_{i},y_j)}{2}.$$
Boundary conditions are strictly enforced with \( u_{0,j}^{(n)}=u_{N_x,j}^{(n)}=u_{i,0}^{(n)}=u_{i,N_y}^{(n)}=0 \) for all \( i \) and \( j \), corresponding to \( u^{(n)} = 0 \) on \(\partial \Omega\). The main steps of the proposed iterative process to solve the direct problem ( \ref{govequation}) are summarized in the following algorithm:

\begin{algorithm}[H]
\begin{enumerate}

\vspace{0.2cm}
\item[{\bf Step 1}] Choose a tolerance parameter $\overline{\epsilon}$ and set $n = 0$. Initialize $u^{(0)}$ as $0$.

\item[{\bf Step 2}] Update the solution by solving problem $(\mathcal{LP}^{(n+1)})$ using the finite difference
method to obtain $u^{(n+1)}$.

\item[{\bf Step 3}] Compute $\|u^{(n+1)}-u^{(n)}\|_{H^1(\Omega)}$. If $\|u^{(n+1)}-u^{(n)}\|_{H^1(\Omega)}\leq \overline{\epsilon}$ go to {\bf Step 5}. Otherwise, proceed to the next step.

\item[{\bf Step 4}] Set  $n \rightarrow n + 1$. Go to {\bf Step 2}.

\item[{\bf Step 5}] The approximated solution $u^{\star}$ is  $u^{(n+1)} $.

\end{enumerate}
    \caption{\it Iterative Approach for Solving Linearized Problem}
    \label{solving-direct}
\end{algorithm}
To evaluate the practical effectiveness and accuracy of the Algorithm \ref{solving-direct}, we apply it to solve the direct problem with a known solution. We set the tolerance parameter to be $\overline{\epsilon}=10^{-6}$ and define the computational domain $\Omega$ as the square region $[0, 1] \times [0, 1]$. The mesh sizes are chosen to be $\Delta x=\Delta y=0.02$. The exact solution we consider is $u_{ex}(x,y)=(x-x^2)\, (y-y^2)$ as showed in Figure \ref{exactt}.

\begin{figure}[H]
    \centering
    \includegraphics[width=8cm,height=6cm]{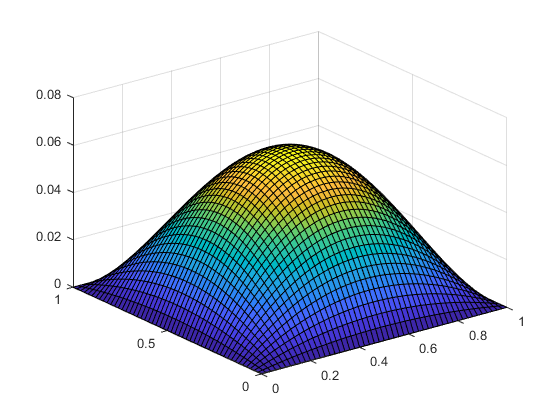}
    \caption{Variation of the exact solution $u_{ex}$}
    \label{exactt}
\end{figure}

\noindent {\bf Test 1:} In this example, we consider the plasticity function as in (\ref{curve}), where $\xi_0^2 = 0.02$ (this value of yield stress is used for soft engineering materials),  $G=42.3$ and strain hardening exponent $\kappa=0.5$. The region $\Omega_1$ denotes the subset of $\Omega$ where $|\nabla u|^2 \leq \xi_0^2$ and  $\Omega_2 = \Omega \setminus \Omega_1 $  (see Figure \ref{regions}). We substitute the exact solution and the plasticity function in the equation $-\nabla.\bigg(g(\vert \nabla u \vert^2) \nabla u \bigg) = F $ and determine the source term accordingly. Figures \ref{plasticity} and \ref{forcing} present the variations of the plasticity function $g(|\nabla u|^2)$ and the source term $F$, respectively.

\begin{figure}[H]
     \centering
     \begin{subfigure}[b]{0.3\textwidth}
         \centering
         \includegraphics[width=\textwidth]{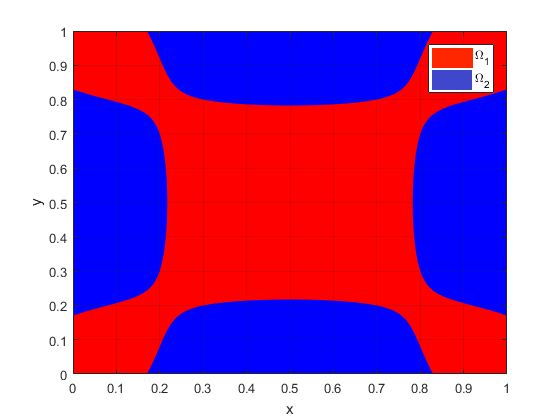}
         \caption{The regions $\Omega_1$ and $\Omega_2$}
         \label{regions}
     \end{subfigure}
     \hfill
     \begin{subfigure}[b]{0.3\textwidth}
         \centering
         \includegraphics[width=\textwidth]{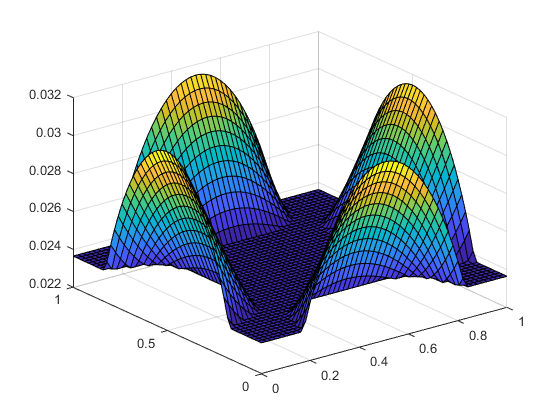}
         \caption{The plasticity function $g$}
         \label{plasticity}
     \end{subfigure}
     \hfill
     \begin{subfigure}[b]{0.3\textwidth}
         \centering
         \includegraphics[width=\textwidth]{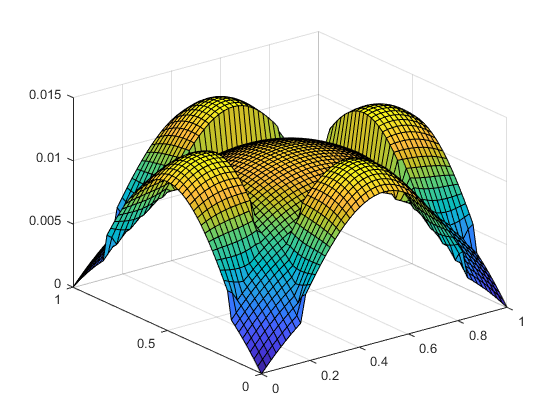}
         \caption{The forcing term $F$}
         \label{forcing}
     \end{subfigure}
        \caption{ Regions $\Omega_1$ and $\Omega_2$ within the domain, along with the plasticity function $g(|\nabla u_{ex}|^2)$ and the forcing term $F$ associated to the exact solution $u_{ex}$ (Test 1).}
        \label{par}
\end{figure}
The numerical application of Algorithm \ref{solving-direct} is further evidenced by the convergence behavior as shown in Figure \ref{convergence}, which presents the variations of the norms $\|u^{(n+1)}-u^{(n)}\|_{H^1}$ and $\|u_{ex}-u^{(n)}\|_{H^1}$ with respect to iteration index $n$. Specifically, Figures \ref{conv_his} and \ref{err} show these variations, respectively, affirming the algorithm’s iterative precision.
\begin{figure}[H]
     \centering
     \begin{subfigure}[b]{0.4\textwidth}
         \centering
         \includegraphics[width=\textwidth]{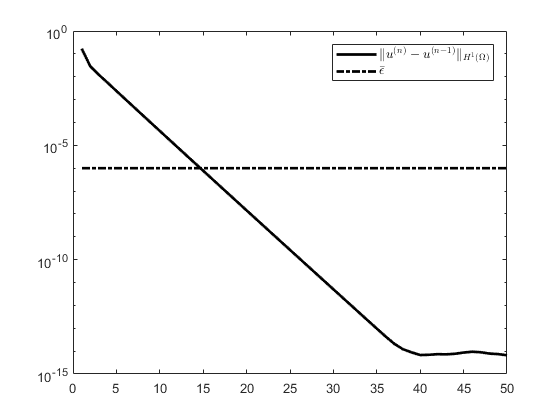}
         \caption{Norm $\|u^{(n+1)}-u^{(n)}\|_{H^1}$}
         \label{conv_his}
     \end{subfigure}
     \hfill
     \begin{subfigure}[b]{0.4\textwidth}
         \centering
         \includegraphics[width=\textwidth]{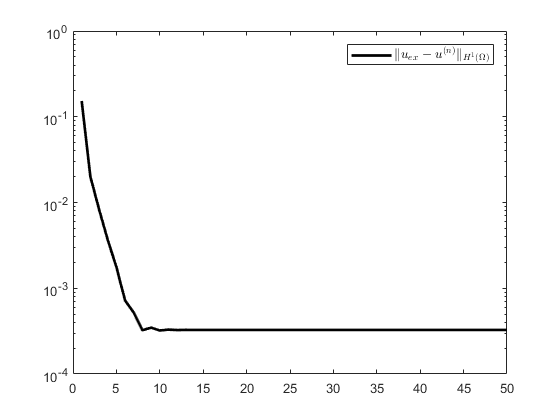}
         \caption{Norm $\|u_{ex}-u^{(n)}\|_{H^1}$}
         \label{err}
     \end{subfigure}
        \caption{Convergence analysis of the iterative process: Norms of successive iterations and error with respect to the exact solution (Test 1).}
        \label{convergence}
\end{figure}
Figure \ref{par2} shows the variation of the numerically approximated solution as well as the absolute errors, highlighting the algorithm's effectiveness and precision.
\begin{figure}[H]
     \centering
     \begin{subfigure}[b]{0.4\textwidth}
         \centering
         \includegraphics[width=\textwidth]{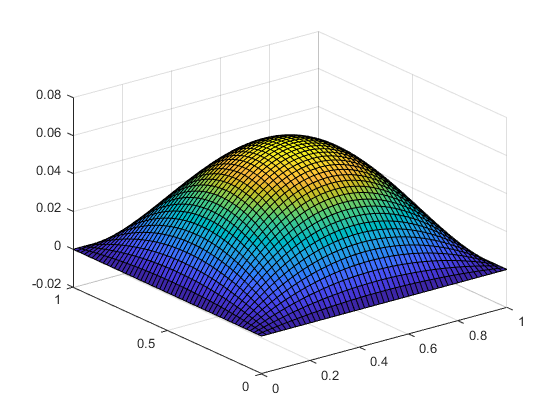}
         \caption{Approximated solution $u^\star$}
         \label{app}
     \end{subfigure}
     \hfill
     \begin{subfigure}[b]{0.4\textwidth}
         \centering
         \includegraphics[width=\textwidth]{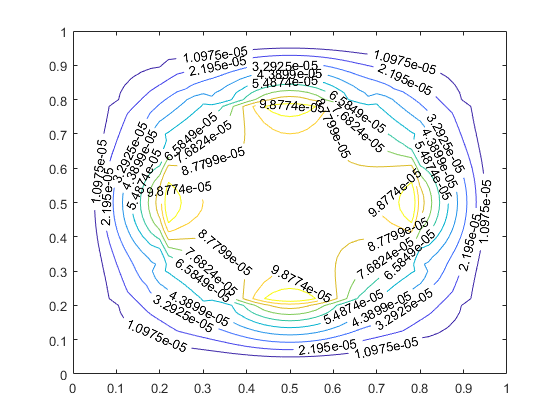}
         \caption{Absolute errors}
         \label{absolute}
     \end{subfigure}
        \caption{The numerically approximated solution alongside the absolute errors  (Test 1).}
        \label{par2}
\end{figure}

\noindent {\bf Test 2:} In this example, the plasticity function \( g(|\nabla u^2) \) is taken as follows:
$$
g(|\nabla u|^2) := \frac{1}{1 + |\nabla u|^2}.
$$
This particular choice of the plasticity function is inspired by image processing, see \cite{perona1990scale}.  Figure \ref{plasticity-ex2} and Figure \ref{forcing-ex2} illustrate the variations in the plasticity function and the source term, respectively.

\begin{figure}[H]
     \centering
     \begin{subfigure}[b]{0.4\textwidth}
         \centering
         \includegraphics[width=\textwidth]{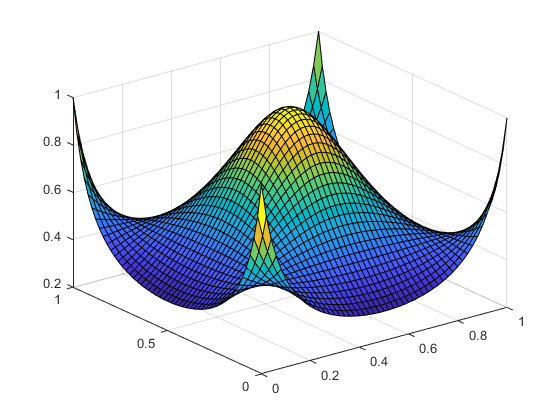}
         \caption{The plasticity function $g$}
         \label{plasticity-ex2}
     \end{subfigure}
     \hfill
     \begin{subfigure}[b]{0.4\textwidth}
         \centering
         \includegraphics[width=\textwidth]{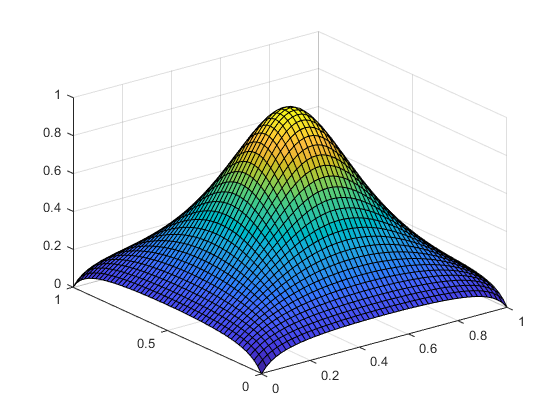}
         \caption{The forcing term $F$}
         \label{forcing-ex2}
     \end{subfigure}
        \caption{ Variations of the plasticity function $g(|\nabla u_{ex}|^2)$ and the forcing term $F$ associated to the exact solution $u_{ex}$  (Test 2).}
        \label{par00}
\end{figure}
Similar to the previous test, we illustrate the efficiency of our proposed algorithm in Figures \ref{conv_succ} and \ref{conv_exa} by plotting the variations of the norms $\|u^{(n+1)}-u^{(n)}\|_{H^1}$ and $\|u_{ex}-u^{(n)}\|_{H^1}$, respectively, with respect to the iteration indices \( n \).
\begin{figure}[H]
     \centering
     \begin{subfigure}[b]{0.4\textwidth}
         \centering
         \includegraphics[width=\textwidth]{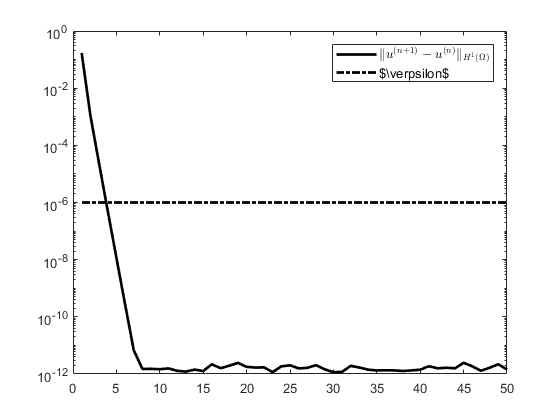}
         \caption{Norm $\|u^{(n+1)}-u^{(n)}\|_{H^1}$}
         \label{conv_succ}
     \end{subfigure}
     \hfill
     \begin{subfigure}[b]{0.4\textwidth}
         \centering
         \includegraphics[width=\textwidth]{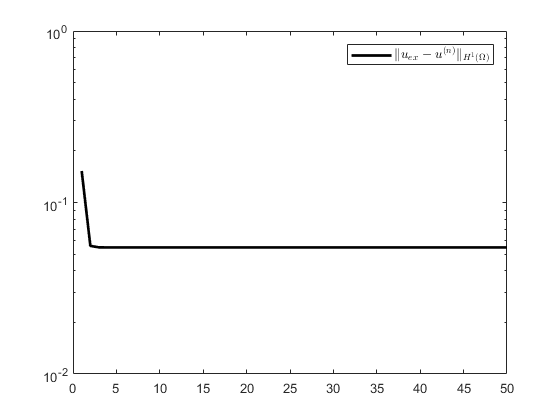}
         \caption{Norm $\|u_{ex}-u^{(n)}\|_{H^1}$}
         \label{conv_exa}
     \end{subfigure}
        \caption{Convergence analysis of the iterative process: Norms of successive iterations and error with respect to the exact solution  (Test 2).}
        \label{convergence2}
\end{figure}

Figure \ref{parex2} indicates the variations in the approximated solution together with the absolute errors, emphasizing the effectiveness and accuracy of the algorithm.

\begin{figure}[H]
     \centering
     \begin{subfigure}[b]{0.4\textwidth}
         \centering
         \includegraphics[width=\textwidth]{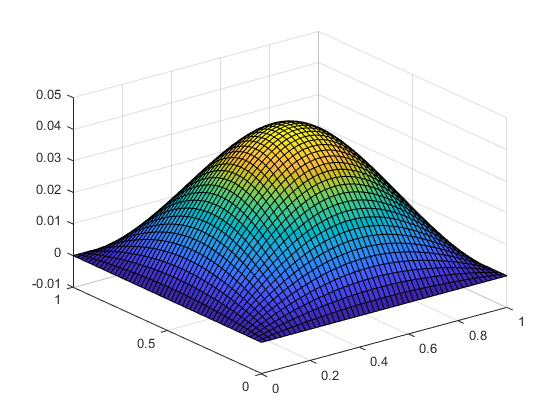}
         \caption{Approximated solution $u^\star$}
         \label{app}
     \end{subfigure}
     \hfill
     \begin{subfigure}[b]{0.4\textwidth}
         \centering
         \includegraphics[width=\textwidth]{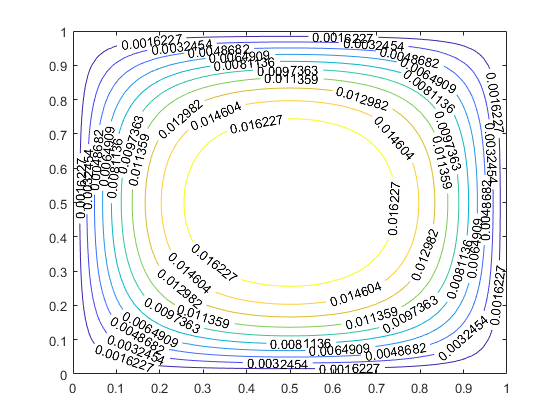}
         \caption{Absolute errors}
         \label{absolute}
     \end{subfigure}
        \caption{The numerically approximated solution alongside the absolute errors (Test 2).}
        \label{parex2} 
\end{figure}
In conclusion, the effectiveness and accuracy of Algorithm \ref{solving-direct} in solving the direct problem (\ref{govequation}) numerically have been rigorously evaluated and demonstrated through the above tests. The findings depicted in Figures \ref{convergence}-\ref{par2} and \ref{convergence2}-\ref{parex2} offer compelling evidence of the algorithm's robust performance. Specifically, Figures \ref{convergence} and \ref{convergence2} illustrate a consistent decrease in the norm $\|u^{(n+1)}-u^{(n)}\|_{H^1}$ with each iteration, indicating rapid convergence to the solution. Similarly, the error measure $\|u_{ex}-u^{(n)}\|_{H^1}$ shown in the same figure highlights the diminishing discrepancy between the  approximated solution and the exact solution, emphasizing the algorithm’s precision. Further affirmation of the algorithm's capability is observed in Figures \ref{par2} and \ref{parex2}, where the approximated solutions and the absolute errors are given. The close alignment between these solutions, coupled with the minimal absolute errors, confirms the high level of accuracy achieved by Algorithm \ref{solving-direct}.

\subsection{Reconstruction Results for the Inverse Problem}
In this subsection, we focus on demonstrating the effectiveness of Algorithm \ref{algo1} in solving the considered inverse problem. Specifically, we apply the algorithm to estimate three key parameters simultaneously: the strain hardening exponent $\kappa$, the yield stress $\xi_0^2$, and the shear modulus $G$. To verify the accuracy of our numerical results, we compute the approximate relative errors as follows:
\[
e_{\kappa_n} = \frac{|\overline{\kappa}_n - \kappa^\dagger|}{|\kappa^\dagger|}, \quad e_{\xi_{0,n}^2} = \frac{|\overline{\xi}_{0,n}^2 - \xi_0^{2,\dagger}|}{|\xi_0^{2,\dagger}|}, \quad e_{G_n} = \frac{|\overline{G}_n - G^\dagger|}{|G^\dagger|}, \quad e_n =\frac{ \|\overline{\theta}_n - \theta^\dagger\|}{\| \theta^\dagger\|},
\]
where $\overline{\theta}_n = (\overline{\kappa}_n, \overline{\xi}_{0,n}^2, \overline{G}_n)$ represents the mean values at the $n$-th iteration in {\bf IREKM}, $\theta^\dagger = (\kappa^\dagger, \xi_0^{2,\dagger}, G^\dagger)$ indicates the exact solution of the inverse problem and $\|\cdot\|$ denotes the Euclidean norm. We define the residual at each iteration $n$ as:
\[
R_n = \|C^{-\frac{1}{2}}(d - \overline{w}_n)\|.
\]

A critical component of iterative algorithms is the implementation of an appropriate stopping rule. We employ the stopping criterion from references \cite{zhang2018bayesian, iglesias2015iterative, iglesias2016regularizing}, which is based on the following condition:
\begin{equation}\label{stop}
R_{n} \leq \tau \delta,   
\end{equation}
where $\tau \approx \frac{1}{\rho}$ (see \cite{iglesias2015iterative}), and $\delta = \|C^{-\frac{1}{2}}(d - \mathbb{F}(\theta^\dagger))\|$ is the noise level in the observational data $\mathcal{T}^\delta = (\mathcal{T}_1^\delta, \dots, \mathcal{T}_M^\delta)$. For our numerical experiments, the initial values $\overline{\kappa}_0$, $\overline{\xi}_{0,0}^2$, and $\overline{G}_{0}$ are randomly selected within predefined ranges, with $N_e = 200$ ensemble members participating. We conduct two test cases encountered in practical applications to test the algorithm's capability. These tests aim to reconstruct the specified material parameters under different conditions, demonstrating the algorithm’s potential for precise material behavior modeling under torsional stress. Our findings will provide insights into the algorithm's robustness and reliability in accurately reflecting the physical properties of materials. For each test, we consider the torsional stress at different angles of twist per unit length $\varphi_i$, for $i=1,\dots,M$. We define $\mathcal{M}_i$ as:
\[
\mathcal{M}_i:=\max_{(x,y)\in \Omega}|\nabla u(x,y;\varphi_i)|^2.
\]
If $\mathcal{M}_i$ exceeds $\xi_0^2$, the measurement data $\mathcal{T}_i$ corresponding to $\varphi_i$ is classified as plastic; otherwise, it is considered elastic. Each scenario tests the algorithm with various numbers and types of measurements. The parameters utilized in our numerical evaluations and their assigned values include setting the computational domain are  $\Omega = [0,1] \times [0,1]$,  $\Delta_x = \Delta_y = 0.02$ and configuring the parameter $\rho$ in Algorithm \ref{algo1} to $0.7$.\\

In the numerical examples below, two class of (soft and stiff) materials described by (\ref{curve}) with parameters $ \bigg(E = 110 \, Gpa \,  or \, \, G = \frac{110}{2 \times (1 + 0.3)}\approx 42.3 \, Gpa, \,  \xi_0^2 = 0.02  \bigg)$ and  $ \bigg(E = 210 \, Gpa \,  or \, \, G = \frac{210}{2 \times (1 + 0.3)}\approx 80.77 \, Gpa, \,  \xi_0^2 = 0.027  \bigg)$ are considered. \\

\vspace{0.2cm}

\noindent \hypertarget{ex1}{{\bf Example 1: Simultaneous determination of the parameters for soft engineering materials.}} In this example, we evaluate the effectiveness of the proposed algorithm for identifying the unknown material parameters in soft materials. We run the algorithm for two different values of $\kappa $ : $\kappa = 0.3$ and $\kappa = 0.7$. 
The probability distributions are set as $\mathbb{P}(\kappa)=U(0.2, 0.9)$, $\mathbb{P}(\xi_0^2)=U(0, 0.15)$, and $\mathbb{P}(G)=U(42, 43)$. For these cases, we analyze the behavior under various angles of twist per unit length $\varphi_i$ (for $i=1,\dots,m$), which are critical for assessing the performance of our algorithm. The selected angles and their corresponding maximum stress values $\mathcal{M}_i$ are given in Table \ref{masex1}.
\begin{table}[H]
    \centering
    \begin{tabular}{ |c|c|c|c|c|c|  }
\hline
\multicolumn{2}{|c|}{Angle $\varphi$}&$\varphi_1=1$&$\varphi_2=0.5$&$\varphi_3=0.1$&$\varphi_4=0.005$ \\
\hline
{\bf Case 1}&$\mathcal{M}_i(\varphi)$&$2.5236$ &$1.3580$&$0.3221$&$0.0175$ \\
\hline
{\bf Case 2}&$\mathcal{M}_i(\varphi)$&$62.473$ &$21.693$&$1.8607$&$0.0175$\\
\hline
\end{tabular}
    \caption{Stress values corresponding to different angles of twist for the two cases ({\bf Example 1})}
    \label{masex1}
\end{table}
To thoroughly test the algorithm, we vary both the number and types of measurements across the two cases. We use different combinations of measurement types—plastic(P) and elastic(E)—across various numbers of angles to explore the algorithm’s adaptability and accuracy in different testing environments. The configurations of these tests are summarized in Table \ref{choixex1}.

\begin{table}[H]
    \centering
    \begin{tabular}{ |c|c|c|c|c|c|c|  }
\hline
$M$&$1$&\multicolumn{2}{|c|}{$2$}&\multicolumn{2}{|c|}{$3$}&$4$ \\
\hline
{\bf Angles}&$\varphi_1$ &$\varphi_1$ \& $\varphi_4$&$\varphi_1$ \& $\varphi_2$&$\varphi_i$, $i=1,2,4$& $\varphi_i$, $i=1,2,3$ &  $\varphi_i$, $i=1,\dots,4$ \\
\hline
{\bf Type} & 1 P  &1 P \& 1 E &2 P &2 P \& 1 E& 3 P &3 P \& 1 E\\
\hline
\end{tabular}
    \caption{Measurement configurations for the {\bf Example 1}}
    \label{choixex1}
\end{table}
The numerical results for the parameters $\theta=(\kappa, \xi_0^2, G)$ in {\bf Example 1} are presented in two scenarios. For {\it Case 1}, the results are detailed in Table \ref{accuracyex1case1}, and for {\it Case 2}, the results are shown in Table \ref{accuracyex1case2}.

  \begin{table}[H]
\small
\centering
\begin{tabular}{|p{1.1cm}|c|c|c|c|c|c|c|}
\hline
Data &$\sigma$& $\overline{\theta}_n$&$e_n$ &$e_{\kappa_n}$& $e_{\xi_{0,n}^2}$ &$e_{G_n}$&$n$\\
\hline
\multirow{ 3}{5em} {\bf 1 Plastic} & $0.0001$ &$(0.227,0.0933,42.53)$&$5.96\mathrm{e}{-3}$&$2.41\mathrm{e}{-1}$&$3.66$&$5.54\mathrm{e}{-3}$&$22$\\
& $0.001$ &$(0.227,0.0933,42.54)$&$6.17\mathrm{e}{-3}$&$2.42\mathrm{e}{-1}$&$3.66$&$5.70\mathrm{e}{-3}$&$21$\\
& $0.01$ &$(0.226,0.0936,42.54)$&$6.18\mathrm{e}{-3}$&$2.45\mathrm{e}{-1}$&$3.68$&$5.72\mathrm{e}{-3}$&$18$\\
\hhline{|=|=|=|=|=|=|=|=|}
\multirow{ 3}{5em} {\bf 1 Plastic  \hspace*{0.5cm} \& \; \; \; \; \;  1 Elastic} & $0.0001$ &$(0.236,0.0801,42.44)$&$3.90\mathrm{e}{-3}$&$2.11\mathrm{e}{-1}$&$3.01$&$3.34\mathrm{e}{-3}$&$21$\\
& $0.001$ &$(0.235,0.0795,42.44)$&$3.91\mathrm{e}{-3}$&$2.15\mathrm{e}{-1}$&$2.97$&$3.36\mathrm{e}{-3}$&$16$\\
& $0.01$ &$(0.231,0.0796,42.44)$&$3.95\mathrm{e}{-3}$&$2.27\mathrm{e}{-1}$&$2.98$&$3.34\mathrm{e}{-3}$&$11$\\
\hhline{|=|=|=|=|=|=|=|=|}
\multirow{ 3}{5em} {\bf 2 Plastic } & $0.0001$ &$(0.300,0.0199,42.34)$&$9.45\mathrm{e}{-4}$&$4.90\mathrm{e}{-4}$&$4.53\mathrm{e}{-3}$&$1.09\mathrm{e}{-3}$&$44$\\
& $0.001$ &$(0.303,0.0188,42.38)$&$1.89\mathrm{e}{-3}$&$1.30\mathrm{e}{-2}$&$5.74\mathrm{e}{-2}$&$1.93\mathrm{e}{-3}$&$30$\\
& $0.01$ &$(0.322,0.0147,42.42)$&$2.88\mathrm{e}{-3}$&$7.50\mathrm{e}{-2}$&$2.64\mathrm{e}{-1}$&$2.94\mathrm{e}{-3}$&$20$\\
\hhline{|=|=|=|=|=|=|=|=|}
\multirow{ 3}{5em} {\bf 2 Plastic  \hspace*{0.5cm} \& \; \; \; \; \;  1 Elastic} & $0.0001$ &$(0.300,0.0198,42.33)$&$7.09\mathrm{e}{-4}$&$9.60\mathrm{e}{-4}$&$5.95\mathrm{e}{-3}$&$8.50\mathrm{e}{-4}$&$37$\\
& $0.001$ &$(0.308,0.0190,42.39)$&$2.13\mathrm{e}{-3}$&$2.98\mathrm{e}{-2}$&$4.64\mathrm{e}{-2}$&$2.14\mathrm{e}{-3}$&$25$\\
& $0.01$ &$(0.314,0.0175,42.42)$&$2.85\mathrm{e}{-3}$&$4.81\mathrm{e}{-2}$&$1.22\mathrm{e}{-1}$&$2.87\mathrm{e}{-3}$&$16$\\
\hhline{|=|=|=|=|=|=|=|=|}
\multirow{ 3}{5em} {\bf 3 Plastic } & $0.0001$ &$(0.300,0.0199,42.35)$&$1.18\mathrm{e}{-3}$&$2.70\mathrm{e}{-4}$&$3.12\mathrm{e}{-3}$&$1.31\mathrm{e}{-3}$&$34$\\
& $0.001$ &$(0.302,0.0194,42.37)$&$1.65\mathrm{e}{-3}$&$9.53\mathrm{e}{-3}$&$2.69\mathrm{e}{-2}$&$1.89\mathrm{e}{-3}$&$24$\\
& $0.01$ &$(0.310,0.0193,42.40)$&$2.37\mathrm{e}{-3}$&$3.50\mathrm{e}{-2}$&$3.39\mathrm{e}{-2}$&$2.49\mathrm{e}{-3}$&$18$\\
\hhline{|=|=|=|=|=|=|=|=|}
\multirow{ 3}{5em} {\bf 3 Plastic  \hspace*{0.5cm} \& \; \; \; \; \;  1 Elastic} & $0.0001$ &$(0.300,0.0199,42.32)$&$7.01\mathrm{e}{-4}$&$1.94\mathrm{e}{-4}$&$1.63\mathrm{e}{-3}$&$7.01\mathrm{e}{-4}$&$36$\\
& $0.001$ &$(0.302,0.0195,42.36)$&$1.57\mathrm{e}{-3}$&$7.07\mathrm{e}{-3}$&$2.25\mathrm{e}{-2}$&$1.57\mathrm{e}{-3}$&$24$\\
& $0.01$ &$(0.308,0.0190,42.41)$&$2.61\mathrm{e}{-3}$&$2.91\mathrm{e}{-2}$&$4.56\mathrm{e}{-2}$&$2.59\mathrm{e}{-3}$&$18$\\
\hline
\end{tabular}
\caption{The numerical results for $\theta=(\kappa, \xi_0^2, G)$ with different noises $\sigma$ in {\bf Example 1} ({\it Case 1})}
\label{accuracyex1case1}
\end{table}

\begin{table}[H]
\small
\centering
\begin{tabular}{|p{1.1cm}|c|c|c|c|c|c|c|}
\hline
Data &$\sigma$& $\overline{\theta}_n$&$e_n$ &$e_{\kappa_n}$& $e_{\xi_{0,n}^2}$ &$e_{G_n}$&$n$\\
\hline
\multirow{ 3}{5em} {\bf 1 Plastic} & $0.0001$ &$(0.634,0.0946,42.54)$&$6.14\mathrm{e}{-3}$&$9.41\mathrm{e}{-2}$&$3.73$&$5.70\mathrm{e}{-3}$&$24$\\
& $0.001$ &$(0.633,0.0947,42.55)$&$6.36\mathrm{e}{-3}$&$9.46\mathrm{e}{-2}$&$3.73$&$6.05\mathrm{e}{-3}$&$23$\\
& $0.01$ &$(0.632,0.0954,42.55)$&$6.37\mathrm{e}{-3}$&$9.57\mathrm{e}{-2}$&$3.77$&$6.10\mathrm{e}{-3}$&$20$\\
\hhline{|=|=|=|=|=|=|=|=|}
\multirow{ 3}{5em} {\bf 1 Plastic  \hspace*{0.5cm} \& \; \; \; \; \;  1 Elastic} & $0.0001$ &$(0.640,0.0830,42.42)$&$3.50\mathrm{e}{-3}$&$8.43\mathrm{e}{-2}$&$3.15$&$3.02\mathrm{e}{-3}$&$23$\\
& $0.001$ &$(0.639,0.0836,42.43)$&$3.71\mathrm{e}{-3}$&$8.59\mathrm{e}{-2}$&$3.18$&$3.14\mathrm{e}{-3}$&$18$\\
& $0.01$ &$(0.635,0.0850,42.43)$&$3.76\mathrm{e}{-3}$&$9.17\mathrm{e}{-2}$&$3.25$&$3.12\mathrm{e}{-3}$&$13$\\
\hhline{|=|=|=|=|=|=|=|=|}
\multirow{ 3}{5em} {\bf 2 Plastic} & $0.0001$ &$(0.699,0.0203,42.24)$&$1.41\mathrm{e}{-3}$&$5.50\mathrm{e}{-4}$&$1.81\mathrm{e}{-2}$&$1.31\mathrm{e}{-3}$&$31$\\
& $0.001$ &$(0.689,0.0260,42.37)$&$1.68\mathrm{e}{-3}$&$1.50\mathrm{e}{-2}$&$2.36$&$1.70\mathrm{e}{-3}$&$20$\\
& $0.01$ &$(0.647,0.0672,42.42)$&$3.29\mathrm{e}{-3}$&$7.49\mathrm{e}{-2}$&$8.25\mathrm{e}{-1}$&$3.00\mathrm{e}{-3}$&$14$\\
\hhline{|=|=|=|=|=|=|=|=|}
\multirow{ 3}{5em} {\bf 2 Plastic  \hspace*{0.5cm} \& \; \; \; \; \;  1 Elastic} & $0.0001$ &$(0.699,0.0203,42.26)$&$9.45\mathrm{e}{-4}$&$3.27\mathrm{e}{-4}$&$1.81\mathrm{e}{-2}$&$1.12\mathrm{e}{-3}$&$32$\\
& $0.001$ &$(0.694,0.0227,42.37)$&$1.66\mathrm{e}{-3}$&$6.22\mathrm{e}{-3}$&$1.37\mathrm{e}{-1}$&$1.47\mathrm{e}{-3}$&$20$\\
& $0.01$ &$(0.678,0.0348,42.42)$&$2.90\mathrm{e}{-3}$&$3.11\mathrm{e}{-2}$&$7.44\mathrm{e}{-1}$&$2.72\mathrm{e}{-3}$&$13$\\
\hhline{|=|=|=|=|=|=|=|=|}
\multirow{ 3}{5em} {\bf 3 Plastic } & $0.0001$ &$(0.699,0.0203,42.32)$&$4.73\mathrm{e}{-4}$&$2.50\mathrm{e}{-4}$&$1.53\mathrm{e}{-2}$&$2.90\mathrm{e}{-4}$&$35$\\
& $0.001$ &$(0.694,0.0221,42.35)$&$1.19\mathrm{e}{-3}$&$3.27\mathrm{e}{-3}$&$1.05\mathrm{e}{-1}$&$1.01\mathrm{e}{-3}$&$24$\\
& $0.01$ &$(0.681,0.0270,42.38)$&$1.95\mathrm{e}{-3}$&$1.67\mathrm{e}{-2}$&$3.54\mathrm{e}{-1}$&$1.78\mathrm{e}{-3}$&$17$\\
\hhline{|=|=|=|=|=|=|=|=|}
\multirow{ 3}{5em} {\bf 3 Plastic  \hspace*{0.5cm} \& \; \; \; \; \;  1 Elastic} & $0.0001$ &$(0.699,0.0200,42.31)$&$2.37\mathrm{e}{-4}$&$1.50\mathrm{e}{-4}$&$1.11\mathrm{e}{-3}$&$5.45\mathrm{e}{-5}$&$31$\\
 & $0.001$ &$(0.695,0.0219,42.35)$&$1.18\mathrm{e}{-3}$&$3.22\mathrm{e}{-3}$&$9.96\mathrm{e}{-2}$&$1.38\mathrm{e}{-3}$&$21$\\
& $0.01$ &$(0.688,0.0269,42.38)$&$1.91\mathrm{e}{-3}$&$1.18\mathrm{e}{-2}$&$3.48\mathrm{e}{-1}$&$1.69\mathrm{e}{-3}$&$15$\\
\hline
\end{tabular}
\caption{The numerical results for $\theta=(\kappa, \xi_0^2, G)$ with different noises $\sigma$ in {\bf Example 1} ({\it Case 2})}
\label{accuracyex1case2}
\end{table}
In the following, we investigate the effectiveness of the iterative regularizing ensemble Kalman method ({\bf IREKM}). In Figure \ref{meansEX1}, we show the approximate solutions for $\theta$ in {\bf Example 1} for case 1 with four measured data ($\varphi_i, 1=1,\dots,4$) and different noise level 
$\sigma$. It can be seen that this example can be recovered very well by using {\bf IREKM}. The logarithmic relative errors
and residuals for this case are shown in Figures \ref{errorsEX1} and \ref{convex1}, we can see that the stopping rule \eqref{stop} is suitable for our inverse problem.
\begin{figure}[H]
     \centering
     \begin{subfigure}[b]{0.3\textwidth}
         \centering
         \includegraphics[width=\textwidth]{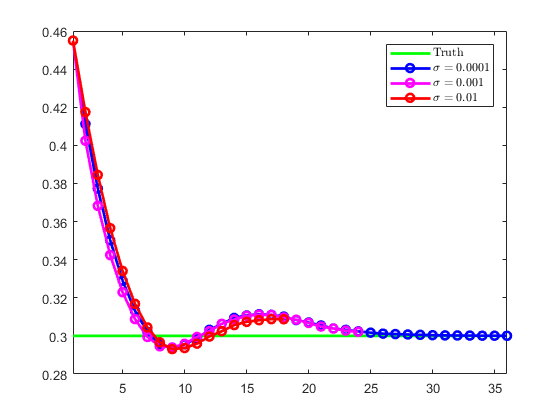}
         \caption{The ensemble means $\overline{\kappa}_n$}
         \label{meank}
     \end{subfigure}
     \hfill
     \begin{subfigure}[b]{0.3\textwidth}
         \centering
         \includegraphics[width=\textwidth]{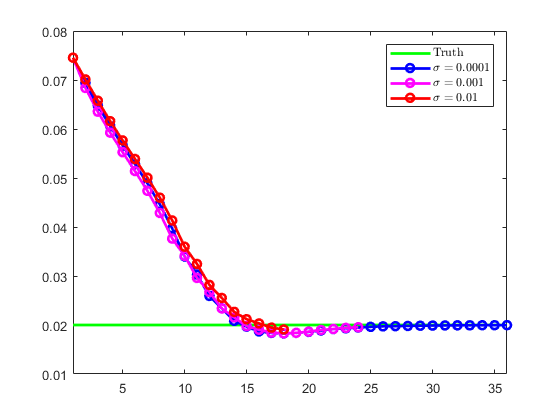}
         \caption{The ensemble means $\overline{\xi}_{0,n}^2$}
         \label{meanT}
     \end{subfigure}
     \hfill
     \begin{subfigure}[b]{0.3\textwidth}
         \centering
         \includegraphics[width=\textwidth]{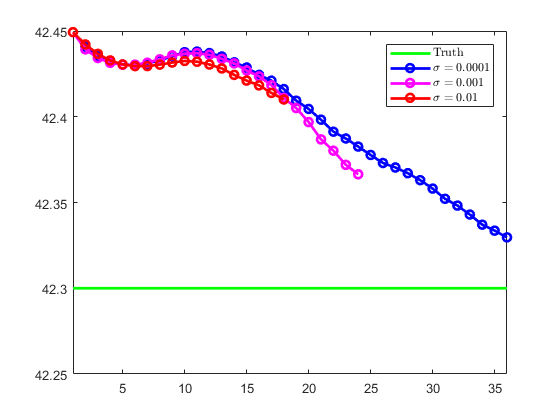}
         \caption{The ensemble means $\overline{G}_n$}
         \label{meanG}
     \end{subfigure}
        \caption{ The ensemble means of $\theta_n$ ({\bf Example 1} - {\it Case 1} - "$M=4$")}
        \label{meansEX1}
\end{figure}

\begin{figure}[H]
     \centering
     \begin{subfigure}[b]{0.3\textwidth}
         \centering
         \includegraphics[width=\textwidth]{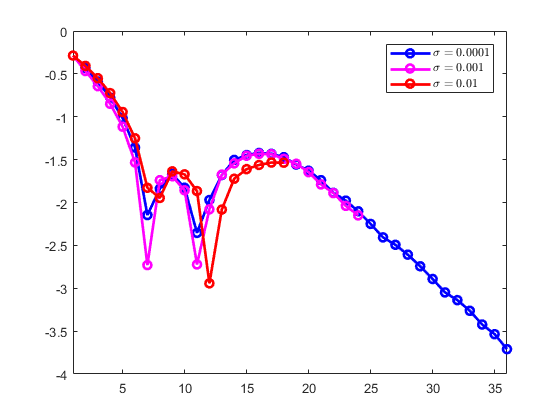}
         \caption{$\log_{10}(e_{\kappa_{n}})$}
         \label{errk}
     \end{subfigure}
     \hfill
     \begin{subfigure}[b]{0.3\textwidth}
         \centering
         \includegraphics[width=\textwidth]{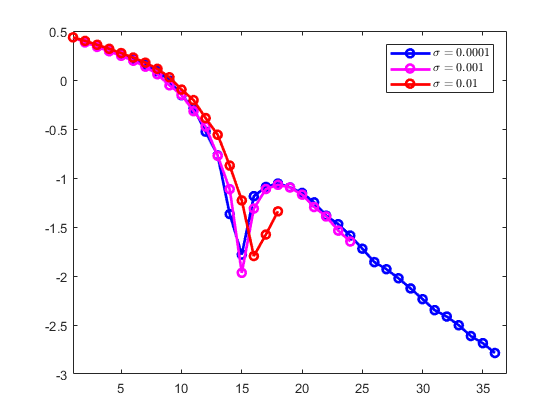}
         \caption{$\log_{10}(e_{\xi_{0,n}^2})$}
         \label{errT}
     \end{subfigure}
     \hfill
     \begin{subfigure}[b]{0.3\textwidth}
         \centering
         \includegraphics[width=\textwidth]{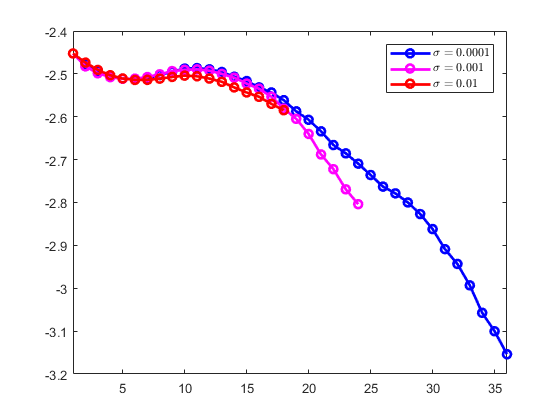}
         \caption{$\log_{10}(e_{G_{n}})$}
         \label{errG}
     \end{subfigure}
        \caption{ Logarithmic relative errors for each unknown term with various noise levels \(\sigma\) ({\bf Example 1} - {\it Case 1} - "$M=4$")}
        \label{errorsEX1}
\end{figure}

\begin{figure}[H]
     \centering
     \begin{subfigure}[b]{0.4\textwidth}
         \centering
         \includegraphics[width=\textwidth]{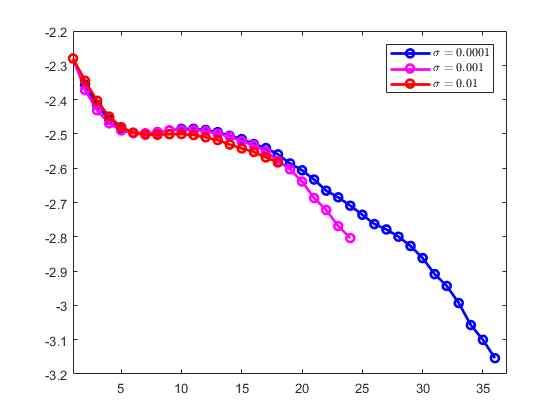}
         \caption{$\log_{10}(e_n)$}
         \label{globex1}
     \end{subfigure}
     \hfill
     \begin{subfigure}[b]{0.4\textwidth}
         \centering
         \includegraphics[width=\textwidth]{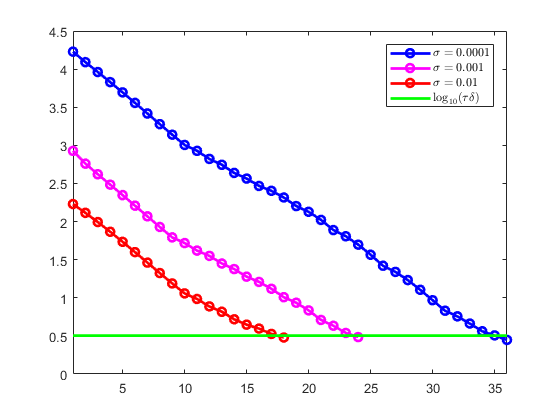}
         \caption{$\log_{10}(R_n)$}
         \label{residuex1}
     \end{subfigure}
        \caption{Logarithmic iteration global relative errors and residuals with various noise levels \(\sigma\) ({\bf Example 1} - {\it Case 1} - "$M=4$")}
        \label{convex1}
\end{figure}

\noindent \hypertarget{ex2}{{\bf Example 2: Simultaneous determination of the parameters for stiff engineering materials.}} In this example, we evaluate the effectiveness of the proposed algorithm for identifying the unknown material parameters in stiff materials. We run the algorithm for two different values of $\kappa $ : $\kappa = 0.3$ and $\kappa = 0.7$. The probability distributions for $\kappa$ and $ {\xi_0}^2$ remain same with the previous test, while $\mathbb{P}(G)=U(80, 81)$. We analyze the impact of the changes in the shift modulus through the algorithm’s responses to various predefined angles of twist $\varphi_i$, with maximum stress values $\mathcal{M}_i$ for each case detailed in Table~\ref{masex2}. 
\begin{table}[H]
    \centering
    \begin{tabular}{ |c|c|c|c|c|c|  }
\hline
\multicolumn{2}{|c|}{Angle $\varphi$}&$\varphi_1=1$&$\varphi_2=0.5$&$\varphi_3=0.1$&$\varphi_4=0.003$ \\
\hline
{\bf Case 1}&$\mathcal{M}(\varphi)$&$5.3118$ &$2.8583$&$0.6780$&$0.0229$ \\
\hline
{\bf Case 2}&$\mathcal{M}(\varphi)$&$179.99$ &$62.500$&$5.3610$&$0.0229$\\
\hline
\end{tabular}
    \caption{Stress values corresponding to different angles of twist for the two cases}
    \label{masex2}
\end{table}
Measurement setups similar to those in Example 1 are applied, with configurations shown in Table~\ref{choixex2}.
\begin{table}[H]
    \centering
    \begin{tabular}{ |c|c|c|c|c|c|c|  }
\hline
Number: $M$&$1$&\multicolumn{2}{|c|}{$2$}&\multicolumn{2}{|c|}{$3$}&$4$ \\
\hline
{\bf Angles}&$\varphi_1$ &$\varphi_1$ \& $\varphi_4$&$\varphi_1$ \& $\varphi_2$&$\varphi_i$, $i=1,2,4$& $\varphi_i$, $i=1,2,3$ &  $\varphi_i$, $i=1,\dots,4$ \\
\hline
{\bf Type} & 1 P  &1 P \& 1 E &2 P &2 P \& 1 E& 3 P &3 P \& 1 E\\
\hline
\end{tabular}
    \caption{Measurement configurations}
    \label{choixex2}
\end{table}
The numerical results for the parameters $\theta=(\kappa, \xi_0^2, G)$ in {\bf Example 2} are presented in Tables \ref{accuracyex2case1} and in Tables \ref{accuracyex2case2} for {\it Case 1} and  {\it Case 2}, respectively.

\begin{table}[H]
\small
\centering
\begin{tabular}{|p{1.1cm}|c|c|c|c|c|c|c|}
\hline
Data&$\sigma$& $\overline{\theta}_n$&$e_n$ &$e_{\kappa_n}$& $e_{\xi_{0,n}^2}$ &$e_{G_n}$&$n$\\
\hline
\multirow{ 3}{5em} {\bf 1 Plastic} & $0.0001$ &$(0.247,0.0962,80.62)$&$2.14\mathrm{e}{-3}$&$1.76\mathrm{e}{-1}$&$2.56$&$1.75\mathrm{e}{-3}$&$20$\\
& $0.001$ &$(0.247,0.0961,80.62)$&$2.14\mathrm{e}{-3}$&$1.75\mathrm{e}{-1}$&$2.56$&$1.75\mathrm{e}{-3}$&$17$\\
& $0.01$ &$(0.247,0.0960,80.62)$&$2.14\mathrm{e}{-3}$&$1.74\mathrm{e}{-1}$&$2.55$&$1.76\mathrm{e}{-3}$&$15$\\

\hhline{|=|=|=|=|=|=|=|=|}
\multirow{ 3}{5em} {\bf 1 Plastic  \hspace*{0.5cm} \& \; \; \; \; \;  1 Elastic} & $0.0001$ &$(0.275,0.0413,80.62)$&$1.89\mathrm{e}{-3}$&$8.10\mathrm{e}{-2}$&$5.32\mathrm{e}{-1}$&$1.73\mathrm{e}{-3}$&$39$\\
& $0.001$ &$(0.247,0.0949,80.63)$&$2.03\mathrm{e}{-3}$&$1.76\mathrm{e}{-1}$&$2.51$&$1.72\mathrm{e}{-3}$&$17$\\
& $0.01$ &$(0.231,0.0949,80.62)$&$2.21\mathrm{e}{-3}$&$2.27\mathrm{e}{-1}$&$2.51$&$1.73\mathrm{e}{-3}$&$11$\\
\hhline{|=|=|=|=|=|=|=|=|}
\multirow{ 3}{5em} {\bf 2 Plastic} & $0.0001$ &$(0.299,0.0271,80.58)$&$2.35\mathrm{e}{-3}$&$3.60\mathrm{e}{-4}$&$5.89\mathrm{e}{-3}$&$2.32\mathrm{e}{-3}$&$36$\\
& $0.001$ &$(0.296,0.0293,80.59)$&$2.22\mathrm{e}{-3}$&$1.29\mathrm{e}{-2}$&$8.80\mathrm{e}{-2}$&$2.13\mathrm{e}{-3}$&$24$\\
& $0.01$ &$(0.293,0.0311,80.56)$&$2.60\mathrm{e}{-3}$&$2.21\mathrm{e}{-2}$&$1.55\mathrm{e}{-1}$&$2.55\mathrm{e}{-3}$&$20$\\
\hhline{|=|=|=|=|=|=|=|=|}
\multirow{ 3}{5em} {\bf 2 Plastic  \hspace*{0.5cm} \& \; \; \; \; \;  1 Elastic} & $0.0001$ &$(0.299,0.0271,80.68)$&$1.11\mathrm{e}{-3}$&$3.30\mathrm{e}{-4}$&$4.67\mathrm{e}{-3}$&$1.02\mathrm{e}{-3}$&$33$\\
& $0.001$ &$(0.298,0.0285,80.66)$&$1.36\mathrm{e}{-3}$&$5.57\mathrm{e}{-3}$&$5.90\mathrm{e}{-2}$&$1.34\mathrm{e}{-3}$&$25$\\
& $0.01$ &$(0.290,0.0374,80.63)$&$1.74\mathrm{e}{-3}$&$3.18\mathrm{e}{-2}$&$3.88\mathrm{e}{-1}$&$1.61\mathrm{e}{-3}$&$18$\\
\hhline{|=|=|=|=|=|=|=|=|}
\multirow{ 3}{5em} {\bf 3 Plastic} & $0.0001$ &$(0.299,0.0270,80.70)$&$8.66\mathrm{e}{-4}$&$2.23\mathrm{e}{-4}$&$3.45\mathrm{e}{-3}$&$7.56\mathrm{e}{-4}$&$32$\\
& $0.001$ &$(0.298,0.0281,80.66)$&$1.36\mathrm{e}{-3}$&$4.10\mathrm{e}{-3}$&$4.28\mathrm{e}{-2}$&$1.25\mathrm{e}{-3}$&$24$\\
& $0.01$ &$(0.294,0.0330,80.65)$&$1.48\mathrm{e}{-3}$&$1.98\mathrm{e}{-2}$&$2.25\mathrm{e}{-1}$&$1.44\mathrm{e}{-3}$&$19$\\
\hhline{|=|=|=|=|=|=|=|=|}
\multirow{ 3}{5em} {\bf 3 Plastic  \hspace*{0.5cm} \& \; \; \; \; \;  1 Elastic} & $0.0001$ &$(0.299,0.0270,80.72)$&$6.19\mathrm{e}{-4}$&$4.40\mathrm{e}{-4}$&$3.45\mathrm{e}{-3}$&$5.74\mathrm{e}{-4}$&$38$\\
& $0.001$ &$(0.296,0.0291,80.67)$&$1.23\mathrm{e}{-3}$&$1.21\mathrm{e}{-2}$&$7.92\mathrm{e}{-2}$&$1.14\mathrm{e}{-3}$&$25$\\
& $0.01$ &$(0.282,0.0388,80.66)$&$1.38\mathrm{e}{-3}$&$5.83\mathrm{e}{-2}$&$4.37\mathrm{e}{-1}$&$1.26\mathrm{e}{-3}$&$18$\\
\hline
\end{tabular}
\caption{The numerical results for $\theta=(\kappa, \xi_0^2, G)$ with different noises $\sigma$ in {\bf Example 2} ({\it Case 1})}
\label{accuracyex2case1}
\end{table}

\begin{table}[H]
\small
\centering
\begin{tabular}{|p{1.1cm}|c|c|c|c|c|c|c|}
\hline
Data &$\sigma$& $\overline{\theta}_n$&$e_n$ &$e_{\kappa_n}$& $e_{\xi_{0,n}^2}$ &$e_{G_n}$&$n$\\
\hline
\multirow{3}{5em} {\bf 1 Plastic} & $0.0001$ &$(0.638,0.0845,80.62)$&$2.13\mathrm{e}{-3}$&$8.91\mathrm{e}{-2}$&$2.13$&$1.73\mathrm{e}{-3}$&$22$\\
& $0.001$ &$(0.658,0.1029,80.62)$&$2.14\mathrm{e}{-3}$&$5.91\mathrm{e}{-2}$&$2.81$&$1.73\mathrm{e}{-3}$&$20$\\
& $0.01$ &$(0.657,0.1037,80.62)$&$2.15\mathrm{e}{-3}$&$6.09\mathrm{e}{-2}$&$2.84$&$1.73\mathrm{e}{-3}$&$17$\\

\hhline{|=|=|=|=|=|=|=|=|}
\multirow{3}{5em} {\bf 1 Plastic \hspace*{0.5cm} \& \; \; \; \; \; 1 Elastic} & $0.0001$ &$(0.648,0.1024,80.62)$&$2.17\mathrm{e}{-3}$&$7.29\mathrm{e}{-2}$&$2.79$&$1.78\mathrm{e}{-3}$&$21$\\
& $0.001$ &$(0.647,0.1038,80.62)$&$2.18\mathrm{e}{-3}$&$7.43\mathrm{e}{-2}$&$2.94$&$1.78\mathrm{e}{-3}$&$18$\\
& $0.01$ &$(0.645,0.1067,80.62)$&$2.21\mathrm{e}{-3}$&$7.74\mathrm{e}{-2}$&$2.95$&$1.80\mathrm{e}{-3}$&$16$\\
\hhline{|=|=|=|=|=|=|=|=|}
\multirow{3}{5em} {\bf 2 Plastic } & $0.0001$ &$(0.699,0.0273,80.66)$&$1.36\mathrm{e}{-3}$&$3.11\mathrm{e}{-4}$&$1.44\mathrm{e}{-2}$&$1.29\mathrm{e}{-3}$&$35$\\
& $0.001$ &$(0.696,0.0301,80.61)$&$1.98\mathrm{e}{-3}$&$4.51\mathrm{e}{-3}$&$1.14\mathrm{e}{-1}$&$1.88\mathrm{e}{-3}$&$23$\\
& $0.01$ &$(0.677,0.0523,80.61)$&$2.02\mathrm{e}{-3}$&$3.19\mathrm{e}{-2}$&$9.37\mathrm{e}{-1}$&$2.20\mathrm{e}{-3}$&$17$\\
\hhline{|=|=|=|=|=|=|=|=|}
\multirow{3}{5em} {\bf 2 Plastic \hspace*{0.5cm} \& \; \; \; \; \; 1 Elastic} & $0.0001$ &$(0.699,0.0273,80.70)$&$8.66\mathrm{e}{-4}$&$3.31\mathrm{e}{-4}$&$1.21\mathrm{e}{-2}$&$7.63\mathrm{e}{-4}$&$35$\\
& $0.001$ &$(0.695,0.0317,80.63)$&$1.73\mathrm{e}{-3}$&$7.10\mathrm{e}{-3}$&$1.74\mathrm{e}{-1}$&$1.66\mathrm{e}{-3}$&$24$\\
& $0.01$ &$(0.677,0.0529,80.63)$&$1.78\mathrm{e}{-3}$&$3.24\mathrm{e}{-2}$&$9.59\mathrm{e}{-1}$&$1.72\mathrm{e}{-3}$&$17$\\
\hhline{|=|=|=|=|=|=|=|=|}
\multirow{3}{5em} {\bf 3 Plastic} & $0.0001$ &$(0.699,0.0272,80.70)$&$8.66\mathrm{e}{-4}$&$1.86\mathrm{e}{-4}$&$8.33\mathrm{e}{-3}$&$7.89\mathrm{e}{-4}$&$34$\\
& $0.001$ &$(0.0302,0.1739,80.65)$&$1.48\mathrm{e}{-3}$&$4.98\mathrm{e}{-3}$&$1.20\mathrm{e}{-1}$&$1.47\mathrm{e}{-3}$&$25$\\
& $0.01$ &$(0.683,0.0440,80.63)$&$1.75\mathrm{e}{-3}$&$2.36\mathrm{e}{-2}$&$6.31\mathrm{e}{-1}$&$1.66\mathrm{e}{-3}$&$18$\\
\hhline{|=|=|=|=|=|=|=|=|}
\multirow{3}{5em} {\bf 3 Plastic \hspace*{0.5cm} \& \; \; \; \; \; 1 Elastic} & $0.0001$ &$(0.699,0.0270,80.76)$&$1.03\mathrm{e}{-4}$&$1.12\mathrm{e}{-4}$&$3.07\mathrm{e}{-3}$&$1.03\mathrm{e}{-4}$&$38$\\
& $0.001$ &$(0.698,0.0287,80.68)$&$1.01\mathrm{e}{-3}$&$2.88\mathrm{e}{-3}$&$6.69\mathrm{e}{-2}$&$1.01\mathrm{e}{-3}$&$26$\\
& $0.01$ &$(0.689,0.0364,80.65)$&$1.37\mathrm{e}{-3}$&$1.44\mathrm{e}{-2}$&$3.49\mathrm{e}{-1}$&$1.36\mathrm{e}{-3}$&$20$\\
\hline
\end{tabular}
\caption{The numerical results for $\theta=(\kappa, \xi_0^2, G)$ with different noises $\sigma$ in {\bf Example 2} ({\it Case 2})}
\label{accuracyex2case2}
\end{table}
In Figure \ref{meansEX2}, we present the approximate solutions for $\theta$ in {\bf Example 2} for {\it Case 2}, incorporating four measured data points ($\varphi_i, i=1,\dots,4$) at various noise levels $\sigma$. This example demonstrates that {\bf IREKM} can effectively recover the desired outcomes. The logarithmic relative errors and residuals for this scenario are displayed in Figures \ref{errorsEX2} and \ref{convex2}, respectively, illustrating that the stopping criterion \eqref{stop} is well-suited for our inverse problem.

\begin{figure}[H]
     \centering
     \begin{subfigure}[b]{0.3\textwidth}
         \centering
         \includegraphics[width=\textwidth]{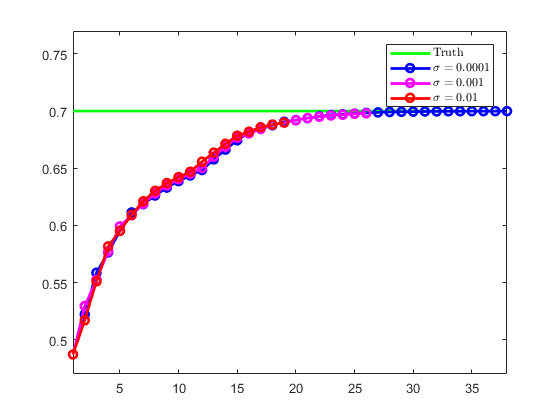}
         \caption{The ensemble means $\overline{\kappa}_n$}
         \label{meank2}
     \end{subfigure}
     \hfill
     \begin{subfigure}[b]{0.3\textwidth}
         \centering
         \includegraphics[width=\textwidth]{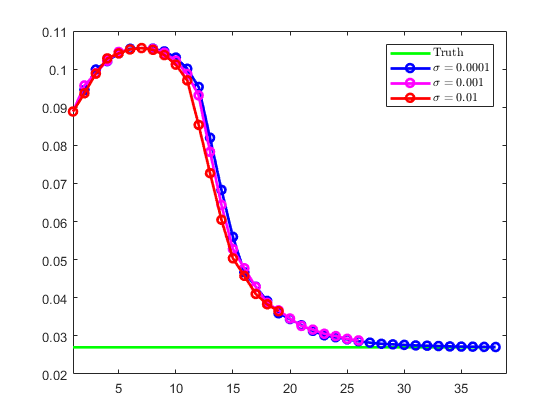}
         \caption{The ensemble means $\overline{\xi}_{0,n}^2$}
         \label{meanT2}
     \end{subfigure}
     \hfill
     \begin{subfigure}[b]{0.3\textwidth}
         \centering
         \includegraphics[width=\textwidth]{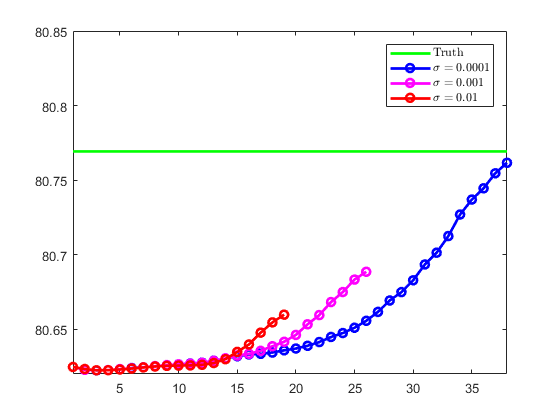}
         \caption{The ensemble means $\overline{G}_n$}
         \label{meanG2}
     \end{subfigure}
        \caption{ The ensemble means of $\theta_n$ ({\bf Example 2} - {\it Case 2} - "$M=4$")}
        \label{meansEX2}
\end{figure}

\begin{figure}[H]
     \centering
     \begin{subfigure}[b]{0.3\textwidth}
         \centering
         \includegraphics[width=\textwidth]{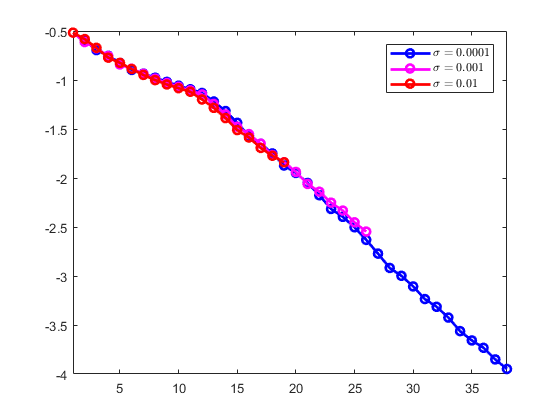}
         \caption{$\log_{10}(e_{\kappa_{n}})$}
         \label{errk2}
     \end{subfigure}
     \hfill
     \begin{subfigure}[b]{0.3\textwidth}
         \centering
         \includegraphics[width=\textwidth]{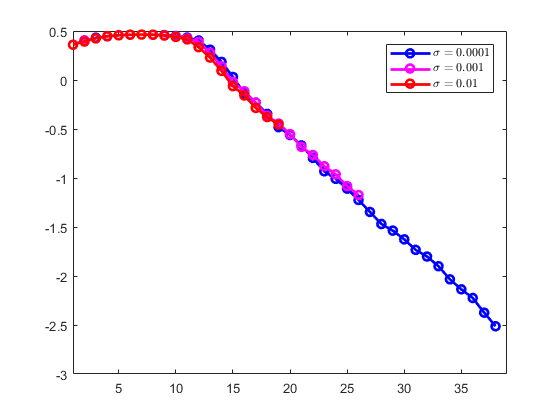}
         \caption{$\log_{10}(e_{\xi_{0,n}^2})$}
         \label{errT2}
     \end{subfigure}
     \hfill
     \begin{subfigure}[b]{0.3\textwidth}
         \centering
         \includegraphics[width=\textwidth]{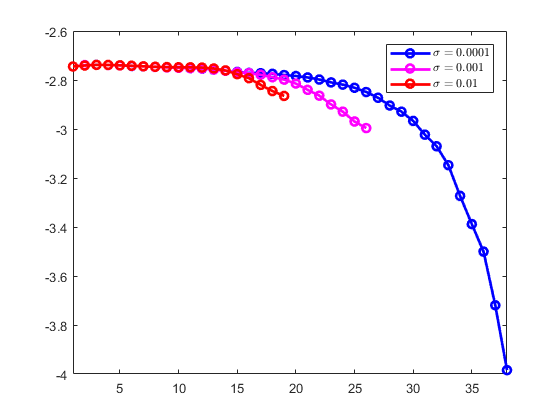}
         \caption{$\log_{10}(e_{G_{n}})$}
         \label{errG2}
     \end{subfigure}
        \caption{  Logarithmic relative errors for each unknown term with various noise levels \(\sigma\) ({\bf Example 2} - {\it Case 2} - "$M=4$")}
        \label{errorsEX2}
\end{figure}

\begin{figure}[H]
     \centering
     \begin{subfigure}[b]{0.4\textwidth}
         \centering
         \includegraphics[width=\textwidth]{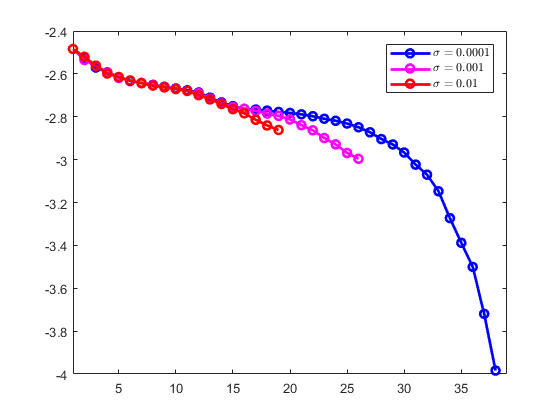}
         \caption{$\log_{10}(e_n)$}
         \label{globex2}
     \end{subfigure}
     \hfill
     \begin{subfigure}[b]{0.4\textwidth}
         \centering
         \includegraphics[width=\textwidth]{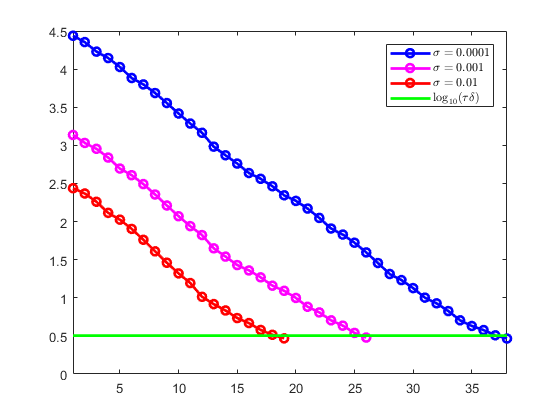}
         \caption{$\log_{10}(R_n)$}
         \label{residuex2}
     \end{subfigure}
        \caption{Logarithmic iteration global relative errors and residuals with various noise levels \(\sigma\) ({\bf Example 2} - {\it Case 2} - "$M=4$")}
        \label{convex2}
\end{figure}

{\bf Concluding remarks:} The numerical study outlined above allows us to draw the following conclusions:

\begin{itemize}
    \item The accuracy of our algorithm progressively improves with an increasing number of measurements, achieving optimal results when using four measurements.
    \item In terms of noise levels, the algorithm performs best with the lowest noise level tested, which is $0.0001$. As the noise level increases to $0.001$ and $0.01$, the accuracy becomes less but still acceptable. 
\end{itemize}

Based on the numerical experiments detailed above, it is evident that {\bf IREKM} is both effective and stable for addressing our inverse problem, which involves the simultaneous recovery of multiple parameters in the elastoplastic torsion of a strain-hardening bar.

\end{document}